\documentclass[final]{siamltex}

\usepackage{latexsym, enumerate}
\usepackage{eepic}
\usepackage{epic}
\usepackage{graphicx}
\usepackage{color}
\usepackage{ifpdf}
\usepackage{amssymb,amsmath,epsfig, algorithm,algorithmicx,multirow}
\usepackage{amsfonts, dsfont}
\usepackage{subfigure}

\newcommand{\ke}{k_{\epsilon}}
\newcommand{\pe}{p_{\epsilon}}
\newcommand{\ue}{u_{\epsilon}}
\newcommand{\se}{\theta_{\epsilon}}
\newcommand{\pK}{\partial K}
\newcommand{\pS}{\partial S}
\newcommand{\we}{w_{\epsilon}^K}
\newcommand{\wh}{w^{*,K}}
\newcommand {\wes}{w_{\epsilon}^S}

\newtheorem{remark}{Remark}[section]
\newtheorem{assumption}{\sc Assumption}[section]

\title{Expanded mixed multiscale finite element methods  and their applications for flows in porous media
\thanks{L. Jiang and J. D. Moulton acknowledge the fund by the Department of Energy at Los Alamos National Laboratory
under contracts  DE-AC52-06NA25396 and the DOE Office of Science
Advanced Computing Research (ASCR) program in Applied Mathematical Sciences. D. Copeland acknowledges the support of Award No. KUS-C1-016-04, made by
King Abdullah University of Science and Technology (KAUST).
}}

\author{L. Jiang\thanks{Applied Mathematics and Plasma Physics,
Los Alamos National Laboratory, NM 87545 ({\tt ljiang@lanl.gov}), Corresponding author.}
\and
 D.  Copeland\thanks{Department of Mathematics, Texas A\&M University,
College Station, TX, 77840 ({\tt dylancopeland@gmail.com}).}
        \and J. D.  Moulton\thanks{Applied Mathematics and Plasma Physics,
Los Alamos National Laboratory, NM 87545
 ({\tt moulton@lanl.gov}). }
}

\begin{document}

\maketitle

\begin{abstract}
We develop a family of expanded mixed Multiscale Finite Element
  Methods (MsFEMs) and their hybridizations for second-order elliptic
  equations.  This formulation expands the standard mixed Multiscale
  Finite Element formulation in the sense that four unknowns (hybrid
  formulation) are solved simultaneously: pressure, gradient of
  pressure, velocity, and Lagrange multipliers.  We use multiscale
  basis functions for both the velocity and the gradient of pressure.
  In the expanded mixed MsFEM framework, we consider both separable
  and non-separable spatial scales. Specifically, we analyze the
  methods in three categories: periodic separable scales,
  $G-$convergent separable scales, and a continuum of scales.  When
  there is no scale separation, using some global information can
  significantly improve the accuracy of the expanded mixed MsFEMs.  We
  present a rigorous convergence analysis of these methods that
  includes both conforming and nonconforming formulations.  Numerical
  results are presented for various multiscale models of flow in
  porous media with shale barriers that illustrate the efficacy of the
  proposed family of expanded mixed MsFEMs.

\end{abstract}



\pagestyle{myheadings}

\thispagestyle{plain}
\markboth{L. Jiang,  D. Copeland and J. D. Moulton}{Expanded mixed MsFEMs}

\section{Introduction}

There are many fundamental and practical problems involving a wide
range of length scales. Typical examples include highly heterogeneous
porous media and composite materials with fine micro-structures.  In
practice, these scales are too fine to treat with direct numerical
simulation as the computational cost far exceeds the capabilities
of computers for the foreseeable future.  As a result, it is a
significant challenge, both theoretically and numerically, to
treat problems with multiple-scales effectively.
A variety of numerical algorithms, varying from upscaling to
multiscale methods, have been developed to capture the influence of
fine-scale spatial heterogeneity on coarse-scale properties of the
solution.  In most upscaling methods the coarse-scale model is
developed by numerically homogenizing (or averaging) the parameters of
the fine-scale model, such as permeability.  The form of the
coarse-scale model is typically assumed to be the same as the
fine-scale model.
Simulations are performed using the coarse-scale
model, and coarse-scale quantities of interest are readily computed.
However, with this parameter upscaling, fine-scale properties of the
solution can no longer be recovered.
In contrast, multiscale methods carry fine-scale
information throughout the simulation, and the coarse-scale
equations are generally not expressed analytically, but rather formed
and solved numerically.

Various numerical multiscale approaches have been developed during the
past decade. A Multiscale Finite Element Method (MsFEM) was introduced
in \cite{hw97} and takes its origin from the pioneering work
\cite{bo83}. Its main idea is to incorporate fine-scale information
into the finite element basis functions and capture their effect on
coarse scales via finite element formulations.  In many cases, the
multiscale basis functions can be pre-computed and used repeatedly in
subsequent computations with different source terms, boundary
conditions and even slightly different coefficients. In some
situations the bases can be updated adaptively.  This leads to a large
computational saving in upscaling multi-phase flows where the pressure
equation needs to be solved many times dynamically.  There are a
number of other multiscale numerical methods with different
approaches, such as residual free bubbles \cite{brezzi00}, the variational
multiscale method \cite{hughes98}, two-scale conservative subgrid
upscaling \cite{arbogast02} and the heterogeneous multiscale method
\cite{ee03}.  Arbogast et al. \cite{apwy07} used a domain
decomposition approach and variational mixed formulation to develop a
multiscale mortar mixed MsFEM.  Jenny et al. \cite{jennylt03} have
used a set of multiscale basis functions similar to \cite{hw97} to
develop a Multiscale Finite Volume Method (MsFV).  A Multilevel
Multiscale Method \cite{lms08} was proposed in the framework of the
Mimetic Finite Difference Method that efficiently evolves a
hierarchy of coarse-scale models.  In recent years,
multiscale methods (e.g., \cite{bo09, cgh10, egw11}) have been
developed to address high-contrast in multiscale coefficients.

The mixed multiscale method was first developed by Arbogast et.al.
\cite{arbogast02} as a locally conservative subgrid upscaling method,
and was later analyzed and extended in \cite{ab06} using a variational
multiscale formulation.  Chen and Hou developed a local multiscale
basis equation for velocity and combined it with a mixed finite
element formulation to propose a mixed MsFEM \cite{ch03}.  The mixed
MsFEMs retain local conservation of mass and have been found to be
useful for solving flow equations in heterogeneous porous media and
other applications.  However, in many practical situations, the
permeability may be very low, and even vanish in some regions of the
domain (e.g., permeability in shale).  In this case its inverse is not
readily available for use in the standard mixed MsFEM formulation, and
hence, the direct application of these methods to practical problems
may be restricted.  To overcome this limitation, we propose a family
of expanded mixed MsFEMs for these multiscale applications.

Expanded Mixed Finite Element Methods (MFEMs) were studied in classic
finite element spaces in the past.  Chen \cite{chen98} developed and
analyzed expanded MFEMs for second-elliptic equations, and obtained
optimal error estimates for linear elliptic equations and certain
nonlinear equations in standard finite element spaces.  Woodward et
al. \cite{wd00} performed an error analysis of an expanded MFEM using
the lowest-order Raviart-Thomas space for Richards equation.  In
\cite{awy97} Arbogast et al. established the connection between the
expanded MFEM and a certain cell-centered finite difference methods.
A dual-dual formulation was introduced in \cite{gh01} for the expanded
MFEM with Raviart-Thomas spaces.
The expanded formulation extends the standard mixed formulation in the
sense that three variables are explicitly approximated, namely, the
scalar unknown (e.g., pressure), its gradient and the velocity.  The
expanded MFEMs are suitable for the cases where the tensor coefficient
of the underlying partial differential equations is small and even
partially vanishing inside cells, which may be viewed as an extreme
case of high-contrast of coefficients.

Porous media formations are often created by complex geological
processes and may contain materials with a widely varying ability to
transmit fluids. The permeability of the porous media may change
dramatically, from almost impermeable barriers to highly permeable
channels.  This high-contrast of the permeability creates significant
challenges for the simulation of flow through porous media. If the
simulation fails to capture the influence of these barriers (e.g.,
shales) and channels (e.g., sand lenses), errors in the quantities of
interest will likely be unacceptably large.  Moreover, if the low
permeability regions are thin with a small inter-barrier spacing, a
fully resolved fine-grid discretization may be difficult to obtain and
costly to apply.  Thus, a multiscale simulation on a coarse grid may
be necessary for these situations.  To simulate the flows on a coarse
grid when the permeability is almost vanishing in some regions, a
reduced-contrast approximation was developed in MsFEM \cite{ce10} to
lower the variance of the coefficients. In \cite{akl06}, Aarnes et al.
proposed an automatic approach to detect the barrier and iteratively
split the coarse cells around barriers to obtain improved solutions in
mixed MsFEM.  For these barrier situations, the proposed expanded
mixed MsFEM can provide an easy and direct approach to perform the
computation.

The purpose of this work is to develop a framework of expanded mixed
MsFEMs and rigorously analyze the convergence for different multiscale
cases.  The work enriches the studies on mixed MsFEMs from previous
works \cite{aej08, akl06,ch03}.  The multiscale phenomena can be
roughly classified into two categories: separable and
non-separable spatial scales.  In the case of separable scales one can
localize the computation of multiscale basis functions. However, these
approaches usually produce resonance errors  Strategies exist to reduce the
resonance errors, such as the oversampling technique introduced in
\cite{hw97}.  Recently, a new technique was proposed in \cite{gl10},
where a zero-th order term is artificially added to the standard
multiscale basis equation \cite{hw97} to make the associated Green's
function decay exponentially, and consequently, the resonance error
can be reduced.  However, these strategies for reducing the
resonance errors usually result in a nonconforming FEM, and moreover,
they do not remedy the poor performance observed when the
there is no scale separation.  Instead, when media
exhibit strong non-separable scales, some global information is needed
for representing non-local effects.  If the global information
captures all relevant information from different scales, then
resonance errors are removed and approximation accuracy is
significantly improved \cite{aej08, bo09, oz07}.  If global
information (or field) is used, we refer to the multiscale methods as
global multiscale methods; otherwise, we refer to them as local
multiscale methods.  We consider the expanded mixed MsFEM for both
separable and non-separable scales.  Multiscale basis functions
are employed for both the velocity and the gradient of the scalar.  We
present convergence analysis for three typical multiscale cases:
periodic highly-oscillatory separable scales,
$G-$convergent separable scales, and a continuum of scales.
For the expanded mixed MsFEM, we consider both
conforming multiscale bases and nonconforming/oversampling multiscale
bases.  The computation of expanded mixed MsFEM is very similar to the
standard mixed MsFEM, but the expanded mixed MsFEM gives more
information about the solution and provides an accurate approximation of
the scalar field's gradient as well. This gradient is often needed in
applications, for example, the pressure gradient is
used for flows with gravity in computing upstream directions.

We consider the proposed expanded MsFEMs in the non-hybrid (standard)
form and the corresponding hybrid form.  The hybridization of
the mixed formulation is of increasing importance to the study of
mixed methods.  Hybridization was initially devised by Fraejis de Veubeke
\cite{frae77} as an efficient implementation technique for mixed finite
elements.  Specifically, hybridization localizes the mass matrix on each
cell, and hence, enables the local elimination of the velocity to obtain a sparse
positive definite system.  Later hybridization was applied to
produce a better approximation of the scalar unknown over each cell.
For example, Arnold and Brezzi \cite{ab85} showed that using the
Lagrange multiplier unknowns introduced by hybridization in a
local post-processing procedure improved the accuracy of the scalar unknown.
In the framework of expanded mixed MsFEMs, we first show the equivalence between
each expanded mixed MsFEM and its hybridization, and then analyze the convergence
of these methods.  In particular, we obtain the convergence rates of the
Lagrange multipliers for each of the three multiscale cases.

The rest of the paper is organized as follows.  Section 2 is devoted
to formulating the expanded mixed formulation for a model elliptic
equation in an abstract framework.  In Section 3, we present
multiscale basis functions and formulations of expanded mixed MsFEM
for separable scales and continuum scales. In Section 4, we present
convergence analysis for both expanded mixed MsFEMs and their
hybridizations. Local multiscale methods and global multiscale methods
are analyzed for the expanded mixed MsFEMs. In Section 5, the expanded
mixed MsFEMs are applied to various multiscale models of flow in
porous media to demonstrate their efficacy.  Finally, some comments
and conclusions are made.


\section{Background}
\label{sect-background}

In this section we highlight the formulation of the expanded MFEM,
and summarize important notation.  We focus on the dual-dual formulation
of the continuum problem because it most naturally supports our analysis.
We review important results of the corresponding discretization as well.

\subsection{The continuum expanded mixed formulation}
\label{sect-cont}

Let $\Omega$ be a bounded domain in $\mathbb{R}^d$, $d=2,3$, with
Lipschitz boundary $\partial \Omega$.  For a subdomain $D$, $m\geq 0$
and $1\leq p\leq \infty$, $W^{m,p}(D)$ and $L^p(D)$ denote the usual
Sobolev space and Legesgue space, respectively.  The norm and seminorm
of $W^{m,p}$ are denoted by $\|\cdot\|_{m,p,D}$ and $|\cdot|_{m,p,D}$,
respectively.  When $p=2$, $W^{m,p}(D)$ is written as $H^{m}(D)$ with
norm $\|\cdot\|_{m, D}$ and seminorm $|\cdot|_{m,D}$.  The norm of
$L^2(D)$ is denoted by $\|\cdot\|_{0, D}$.  We also use the following
spaces
\begin{eqnarray*}
\begin{split}
H^1_0(D)&:=\{q\in H^1(D): q|_{\partial D}=0\}\\
H(div, D)&:=\{ v\in [L^2(D)]^d:  \|v\|_{div, D}:=\|v\|_{0, D} +\|div (v)\|_{0,D}< \infty \}\\
H^0(div, D)&:=\{v\in H(div, D):  v\cdot n|_{\partial D}=0\}.
\end{split}
\end{eqnarray*}
For a normed space $X$ on $\Omega$, $\|\cdot\|_{X}$ denotes the
underlying norm on $X$.  We denote $X(D)$ to be the restriction of $X$
to subdomain $D$.  In the paper, $(\cdot, \cdot)$ is the usual
$L^2(\Omega)$ inner product, and $(\cdot, \cdot)_D$ is the $L^2(D)$
inner product.

We consider the following elliptic problem
\begin{eqnarray}
\label{target-elliptic}
\begin{cases}
\begin{split}
 -div (k\nabla p) &=f(x) \quad \text{in} \quad\Omega\\
  p &=0 \quad \text{on} \quad \partial\Omega.
\end{split}
\end{cases}
\end{eqnarray}
%
We may assume $k$ is symmetric positive definite.
Let $\theta=\nabla p$ and $u=-k\theta$.  The equation
(\ref{target-elliptic}) can be rewritten as
\begin{eqnarray}
\label{1st-order-target}
\begin{cases}
\begin{split}
u +k \theta &=0\\
\theta-\nabla p &=0\\
div(u)& =f.
\end{split}
\end{cases}
\end{eqnarray}
Here $u$ usually refers to the fluid velocity. We define the following function
spaces for solutions:
\[
X_1=[L^2(\Omega)]^d, \quad X_2=H(div, \Omega), \quad X=X_1 \times X_2,   \quad  Q=L^2(\Omega).
\]
The expanded mixed formulation of (\ref{target-elliptic})   reads:  find $(\theta, u, p)$ in $X_1 \times X_2 \times Q$,
satisfying
\begin{eqnarray}
\label{exp-fem-cont}
\begin{cases}
\begin{split}
(k\theta, \xi)+(u, \xi)&=0   \quad \forall \xi\in X_1\\
(\theta, \tau)+(p, div\tau)&=0  \quad  \forall \tau\in X_2\\
(div u, q)&=(f, q)  \quad q\in Q.
\end{split}
\end{cases}
\end{eqnarray}
We have the following theorem, which addresses the relation between the solution
of (\ref{target-elliptic}) and the solution of (\ref{exp-fem-cont}).
\begin{theorem} (See Theorem 3.5 in \cite{chen98}.)
\label{exp-elliptic-thm}
If $\big((\theta, u), p \big)\in X \times Q$ is the solution of
(\ref{exp-fem-cont}), then $p\in H_0^1(\Omega)$ is the solution of
(\ref{target-elliptic}), satisfying $\theta=\nabla p$ and $u=-k
\theta$.  Conversely, if $p\in H_0^1(\Omega)$ is the solution of
(\ref{target-elliptic}), then (\ref{exp-fem-cont}) has the solution
$\big((\theta, u), p \big)\in X \times Q$ satisfying $\theta=\nabla p$
and $u=-k \theta$.
\end{theorem}

For theoretical analysis, we introduce the operators associated with
(\ref{exp-fem-cont}).  Let $A_1: X_1 \longrightarrow X'_1$ and $B_1:
X_1\longrightarrow X'_2$, where the notation $'$ denotes the dual of a
space.  These are defined, respectively, by
\begin{eqnarray*}
\begin{split}
[A_1(\theta), \xi]&:=(k\theta, \xi),  \quad \forall \xi\in X_1 \, , \\
[B_1(\theta), v]&:=(\theta, v), \quad \forall v\in X_2 \, ,
\end{split}
\end{eqnarray*}
from which we define $A: X \longrightarrow X'$  by
\[
  [A(\theta, u), (\xi, v)]:=[A_1(\theta), \xi]+[B_1^*(u),\xi]+[B_1(\theta), v] \, ,
\]
where $B_1^*$ is the transpose operator of $B_1$.  Let $O$ denote the
null operator. Then $A$ can be rewritten as
\[
A=
\left[ \begin{array}{cc}
A_1 &  B_1^* \\
B_1 & O
\end{array} \right].
\]

Let $B: X_2\longrightarrow Q'$ and $F\in Q'$.  They are defined, respectively, by
\begin{eqnarray*}
\begin{split}
[B(u), q]&:=(div u, q),  \quad \forall q\in Q\\
[F, q]&:=(f, q),   \forall q\in Q.
\end{split}
\end{eqnarray*}
Using this operator notation, the expanded mixed formulation
(\ref{exp-fem-cont}) is equivalent to finding $\big((\theta, u), p
\big)$ in $X \times Q$ such that
\begin{equation}
\label{operator-eq}
\begin{bmatrix}
 A_1 & B_1^* & O \\
B_1 & O &  B^* \\
O & B & O
\end{bmatrix}
 \left[ \begin{array}{c} \theta \\ u \\ p \end{array} \right]
=\left[\begin{array}{c} O\\ O\\ F \end{array}\right].
\end{equation}
Equation (\ref{operator-eq}) is called a dual-dual mixed formulation
of (\ref{exp-fem-cont}) \cite{ga99}, since the operator $A$ itself has
a similar dual-type structure.  Because the operator $B$ satisfies a
continuous inf-sup condition \cite{BreFort91} and the operator $B_1$
satisfies a continuous inf-sup condition on $\ker (B)$, the system
(\ref{operator-eq}) has a unique solution $\big((\theta, u), p \big)$
in $X \times Q$. Moreover, there exists a $C>0$, independent of the
solution, such that \cite{gh01}
\begin{eqnarray}
  \|\big( (\theta, u), p)\|_{X \times Q} \leq C \|f\|_{0, \Omega} \, ,
\end{eqnarray}
where
\[
  \|\big( (\theta, u), p)\|_{X \times Q}:=\|\theta\|_{X_1}+\|u\|_{X_2}+\|p\|_{Q}:=\|\theta\|_{0, \Omega}+\|u\|_{div, \Omega}+\|p\|_{0, \Omega} \, .
\]

\subsection{The discrete expanded mixed formulation}
\label{sect-discrete}

Let $A_{1,h}$, $B_{1,h}$ and $B_h$ be approximations of operators
$A_1$, $B_1$ and $B$, respectively. Let $F_h$ be an approximation of
$F$.  Then the numerical formulation of (\ref{operator-eq}) reads
\begin{equation}
\label{operator-num-eq}
\begin{bmatrix}
 A_{1,h} & B_{1,h}^* & O \\
B_{1,h} & O &  B^*_h \\
O & B_h & O
\end{bmatrix}
 \left[ \begin{array}{c} \theta_h \\ u _h\\  p_h \end{array} \right]
=\left[\begin{array}{c} O\\ O\\ F_h \end{array}\right].
\end{equation}

Let $[\cdot , \cdot]_h$ be a numerical inner product. Let $X_{1,h}$,
$X_{2,h}$ and $Q_h$ be finite dimensional approximation of $X_1$,
$X_2$ and $Q_h$, respectively.  Set $X_h:=X_{1,h}\times X_{2,h}$.
Then we have the following abstract result for well-posedness of
(\ref{operator-num-eq}).
\begin{lemma}\cite{ga99}
\label{inf-sup-abstract-lem}
Assume that \\
(1) there exists a positive constant $C_1$ independent of $h$ such
that for any $q_h\in Q_h$,
\begin{equation}
\sup_{v_h \in X_{2,h}\setminus \{0\}} \frac{[B_h(v_h), q_h]_h}{\|v_h\|_{X_2}}\geq C_1 \|q_h\|_{Q};
\end{equation}
(2) there exists a positive constant $C_2$ independent of $h$ such
that for any $v_h \in \ker(B_h)$
\begin{equation}
  \sup_{\xi_h \in X_{1,h}\setminus \{0\}}
    \frac{[B_{1,h}(\xi_h), v_h]_h}{\|\xi_h\|_{X_1}} \geq C_2 \|v_h\|_{X_2} ;
\end{equation}
(3) there exist positive constant $C_3$ and constant $C_4$ independent
of $h$ such that
\[
   C_3 \|\theta_h\|_{X_1}^2 \leq  [A_{1,h}(\theta_h), \theta_h]_h \leq  C_4  \|\theta_h\|_{X_1}^2
   \quad \text{for} \quad \forall \theta_h\in X_{1,h}.
\]

Then there exists a unique $(\theta_h, u_h, p_h)\in X_{1,h}\times
X_{2,h} \times Q_h$ solution of (\ref{operator-num-eq}).
\end{lemma}

We note that $\|\cdot\|_{X_2}$ should be an element broken norm,
provided that $X_{2,h}$ is a nonconforming space in $X_2$, i.e,
$X_{2,h}\not\subset X_2$.  We have the following Strang-type lemma for
the problem (\ref{operator-num-eq}).
\begin{lemma}\cite{ga99}
\label{strang-lem}
Let $\big( (\theta, u), p\big)\in X\times Q$ and
$\big( (\theta_h,u_h), p_h)\in X_h \times Q_h$
be the unique solutions of (\ref{operator-eq})
and (\ref{operator-num-eq}), respectively.  Then
there exists $C=C(\|B_1\|, \|B\|, C_1, C_2, C_3, C_4)>0$,
where $C_1$, $C_2$, $C_3$ and $C_4$ are defined in Lemma
\ref{inf-sup-abstract-lem}, such that
\begin{eqnarray}
\label{strang-ineq}
\begin{split}
&\|\big( (\theta, u), p\big)- \big( (\theta_h, u_h), p_h\big)\|_{X\times Q} \leq \\
&C\big \{  \inf_{((\xi_h, v_h), q_h)\in X_h \times Q_h} \|\big( (\theta, u), p\big)- \big( (\xi_h, v_h), q_h\big)\|_{X\times Q}\\
&+\sup_{v_h\in X_{2,h}\setminus\{0\}} \frac{[-B_{1,h} (\theta)-B_{h}^*(p), v_h]_h}{\|v_h\|_{X_2}} \big\}.
\end{split}
\end{eqnarray}
\end{lemma}

Again, $\|\cdot\|_{X_2}$ in Lemma \ref{strang-lem} should be an
element broken norm, provided that $X_{2,h}$ is a nonconforming space
in $X_2$.  The second term on the right hand side of
(\ref{strang-ineq}) is the so-called consistency error.  It is easy to
check that
\[
\sup_{v_h\in X_{2,h}\setminus\{0\}} \frac{[-B_{1,h} (\theta)-B_{h}^*(p), v_h]_h}{\|v_h\|_{X_2}} =\sup_{\tilde{v}_h\in \ker(B_h) \setminus\{0\}} \frac{[-B_1 (\theta), \tilde{v}_h]}{\|\tilde{v}_h\|_{X_2}}.
\]
Consequently, if $\ker(B_{h})\subset \ker{B}$, then
\[
\sup_{v_h\in X_{2,h}\setminus\{0\}} \frac{[-B_{1,h} (\theta)-B_{h}^*(p), v_h]_h}{\|v_h\|_{X_2}}=0,
\]
which is the case for a conforming expanded mixed FEM.

As a general remark, the generic constant $C$ is assumed to be
independent of the mesh diameter $h$ throughout the paper.

\section{Formulations of expanded mixed MsFEMs}

In this section, we describe how to construct multiscale basis functions
for expanded mixed MsFEMs.

Let $\mathfrak{T}_h$ be a quasi-uniform partition of $\Omega$ and $K$
be a representative coarse mesh with $diam(K)=h_K$.  Let $h=\max\{h_K,
K\in \mathfrak{T}_h\}$.  Let $(V_h, Q_h)$ be a classic mixed finite
element space pair  such as the Raviart-Thomas,
Brezzi-Douglas-Marini or Brezzi-Douglas-Fortin-Marini spaces
(cf. \cite{BreFort91}).  In this
section, we present multiscale basis functions for standard local
expanded mixed MsFEM, oversampling expanded mixed MsFEM, and global
expanded mixed MsFEM.

\subsection{Local expanded mixed MsFEM}
\label{local-ExMMsFEM}

Let $V_h(K):=V_h|_K$ (the finite element velocity space localized on
$K$) and $\chi^K\in V_h(K)$ be a velocity basis function defined on
$K$.  Following the mixed MsFEMs developed in \cite{ch03}, we define
the multiscale basis equation for the local expanded mixed MsFEM in
the following way:
\begin{eqnarray}
\label{local-basis}
\begin{cases}
\begin{split}
\psi_{\chi}^K+k \eta_{ \chi}^K &=0  \quad \text{in} \quad K\\
\eta_{\chi}^K -\nabla \phi_{ \chi}^K &=0 \quad \text{in} \quad K\\
div(\psi_{\chi}^K) &= div \chi^K  \quad \text{in} \quad K\\
\psi_{\chi}^K\cdot n_{\pK} &=\chi^K \cdot n_{\pK}   \quad \text{on} \quad \pK,
\end{split}
\end{cases}
\end{eqnarray}
where $n_{\pK}$ is the outward unit normal to $\pK$.  We can obtain
different mixed MsFEM basis functions corresponding to different
choices of $\chi^K$.  Chen and Hou \cite{ch03} choose $\chi^K$ to be
the lowest Raviart-Thomas basis.  The weak expanded mixed formulation
of (\ref{local-basis}) is to find $\big( (\eta_{\chi}^K,
\psi_{\chi}^K), \phi_{\chi}^K\big) \in X(K) \times Q(K)/R$ such that
\begin{eqnarray}
\label{weak-basis-eq}
\begin{cases}
\begin{split}
(k \eta_{\chi}^K, \xi) + (\psi_{\chi}^K, \xi)&=0,  \quad \forall\xi\in [L^2(K)]^d \\
(\eta_{\chi}^K, \tau)+(\phi_{\chi}^K,  div \tau)&=0,  \quad \forall \tau \in H^0(div, K)\\
(div \psi_{\chi}^K,  q)&=(div \chi^K, q),   \quad \forall q \in L^2 (K)
\end{split}
\end{cases}
\end{eqnarray}
with $\psi_{\chi}^K\cdot n_{\pK}=\chi^K \cdot n_{\pK}$ on $\pK$.
The function $\eta_{\chi}^K$ is a basis function for the gradient
variable $\theta$ and $\psi_{\chi}^K$ is a basis function for velocity
$u$.  We define finite element spaces for the local expanded mixed
MsFEM as follows:
\begin{eqnarray*}
\begin{split}
X_{1,h}^l&:=\{\theta \in [L^2(\Omega)]^d:  \theta|_K\in span\{\eta_{\chi}^K\} \quad \text{for each $K\in \mathfrak{T}_h$} \}\\
X_{2,h}^l&:=\{u \in H(div, \Omega):  u|_K\in span\{\psi_{\chi}^K\} \quad \text{for each $K\in \mathfrak{T}_h$} \}\\
X_h^l&:=X_{1,h}^l\times X_{2,h}^l.
\end{split}
\end{eqnarray*}
The finite element space for pressure in the expanded mixed MsFEM is
$Q_h$, which is the same as in the classic mixed finite element pairs.
On each coarse element $K$, the velocity basis functions are
associated with its faces. Therefore, velocity basis functions of
$X_{2,h}^l$ have the support of two coarse elements sharing a common
interface. However, the basis functions of $X_{1,h}^l$ are supported
in only one coarse element. Each basis function of $X_{1,h}^l$ is
associated with a velocity basis function via $\psi_{\chi}^K=-k
\eta_{\chi}^K$, but it is restricted to only one coarse element K.

The local expanded mixed MsFEM formulation of the global
problem (\ref{exp-fem-cont}) reads: find $\big((\theta_h^l, u_h^l), p_h^l\big)\in X_{h}^l \times
Q_h$ such that
\begin{eqnarray}
\label{exp-fem-num-local}
\begin{cases}
\begin{split}
(k\theta_h^l, \xi_h)+(u_h^l, \xi_h)&=0   \quad \forall \xi_h\in X_{1,h}^l\\
(\theta_h^l, \tau_h)+(p_h^l, div\tau_h)&=0  \quad  \forall \tau_h\in X_{2,h}^l\\
(div u_h^l, q_h)&=(f, q_h)  \quad q\in Q_h
\end{split}
\end{cases}
\end{eqnarray}
subject to the global boundary conditions.

It is well known \cite{chen98} that the linear system of equations in
the expanded mixed FEM produces an indefinite matrix, which is a
considerable source of trouble when solving the linear system.  We can
introduce Lagrange multipliers on interfaces of cells and localize
mass matrix to each element to obtain a sparse symmetric positive
definite system by elimination, which is suitable for many linear
solvers.  This will give rise to a hybrid formulation.  To this end we
require some further notations.

For the hybrid formulation we define the MsFEM velocity space to be
$\tilde{X}_{2,h}^l$, whose normal components are not necessarily
continuous on $\mathcal{F}_h$, the collection interior interfaces of
$\mathfrak{T}_h$. They are defined as follows:
\begin{eqnarray*}
\begin{split}
\tilde{X}_{2,h}^l&:=\{u \in [L^2(\Omega)]^d:  u|_K\in span\{\psi_{\chi}^K\} \quad \text{for each $K\in \mathfrak{T}_h$} \}\\
\end{split}
\end{eqnarray*}
We note that basis functions in $\tilde{X}_{2,h}^l$ do not require
normal continuity. The basis function in $\tilde{X}_{2,h}^l$ is
supported on a coarse grid, and this is different from the basis
function in $X_{2,h}^l$ whose support is two coarse cells sharing the
common face.  We note that
\[
X_{2,h}^l=\tilde{X}_{2,h}^l\cap H(div, \Omega) .
\]

The finite element space for the Lagrange multiplier is defined as
\begin{eqnarray*}
\Pi_h^l:=\big\{\pi\in L^2(\mathcal{F}_h): \pi|_{e}\in X_{2,h}^l\cdot n_{e}  \quad \text{for each  $e\in \mathcal{F}_h$}, \quad \pi|_e=0 \quad \text{for $e\subset \partial \Omega$}  \big\} .
\end{eqnarray*}

Define  the operator $C_h^l: \tilde{X}_{2,h}^l\longrightarrow (\Pi_h^l)'$ by
\begin{equation}
\label{def-C}
[C_h^l(\tau_h), \pi_h]=\sum_{K} (\tau_h\cdot n_{\pK}, \pi_h)_{\pK},  \quad  \forall \tau_h\in \tilde{X}_{2,h}^l, \pi_h\in \Pi_h^l.
\end{equation}
The hybridization of local expanded mixed MsFEM formulation of
(\ref{exp-fem-num-local}) reads: find $(\bar{\theta}_h^l, \bar{u}_h^l,
\bar{p}_h^l, \lambda_h^l )\in X_{h}^l \times \tilde{X}_{2,h}^l \times
Q_h\times \Pi_h^l$ such that
\begin{eqnarray}
\label{exp-fem-hybrid-local}
\begin{cases}
\begin{split}
(k\bar{\theta}_h^l, \xi_h)+(\bar{u}_h^l, \xi_h)&=0   \quad \forall \xi_h\in X_{1,h}^l\\
(\bar{\theta}_h^l, \tau_h)+\sum_K (\bar{p}_h^l, div\tau_h)&=[C_h^l(\tau_h), \lambda_h^l]  \quad  \forall \tau_h\in \tilde{X}_{2,h}^l\\
\sum_K (div \bar{u}_h^l, q_h)&=(f, q_h)  \quad q\in Q_h\\
[C_h^l(\bar{u}_h^l), \pi_h]&=0   \quad \pi_h\in \Pi_h^l.
\end{split}
\end{cases}
\end{eqnarray}

\subsection{Oversampling expanded mixed MsFEM}
\label{os-ExMMsFEM}

If the equation (\ref{local-basis}) is solved in a block $S$ larger
than $K$, and the interior information of the solution is taken to
construct the basis functions in $K$, then this results in the
oversampling technique introduced in \cite{hw97, ch03}.  We note that
oversampling MsFEM is a modified local MsFEM.  Using the oversampling
can reduce resonance error.

Here we follow the outline in \cite{ch03} to present the oversampling
multiscale basis functions.  Let $\big ( (\eta_{\chi}^S,
\psi_{\chi}^S), \phi^S_{\chi}\big)\in X(S) \times Q(S)/R$ be the
solution of the equation
\begin{eqnarray}
\label{os-basis}
\begin{cases}
\begin{split}
 \psi_{\chi}^S+k \eta_{ \chi}^S &=0  \quad \text{in} \quad S\\
\eta_{\chi}^S -\nabla \phi_{ \chi}^S &=0 \quad \text{in} \quad S\\
div(\psi_{\chi}^S) &= div \chi^S  \quad \text{in} \quad S\\
\psi_{\chi}^S\cdot n_{\pK} &=\chi^S \cdot n_{\pK}   \quad \text{on} \quad \partial S.
\end{split}
\end{cases}
\end{eqnarray}
Let $j$ and $l$ be indices of faces of $K$.   We define
\begin{equation}
\label{def-bar-psi}
\bar{\psi}_{\chi_j}^K:=\sum_{l} c_{jl} \psi_{\chi_l}^S|_K,  \quad \bar{\eta}_{\chi_j}^K:=\sum_{l} c_{jl} \eta_{\chi_l}^S|_K,
\end{equation}
where the constants $c_{jl}$ are chosen such that
\begin{equation}
\label{const-c-os}
\chi_j^K=\sum_l c_{jl}\chi_l^S|_K.
\end{equation}
We define
\begin{eqnarray*}
\tilde{X}_{2,h}^{os}&:=\{u \in [L^2(\Omega)]^d:  u|_K\in span\{\bar{\psi}_{\chi}^K\} \quad \text{for each $K\in \mathfrak{T}_h$} \}.
\end{eqnarray*}

We introduce an operator $\mathcal{M}_h: \tilde{X}_{2,h}^{os}
\longrightarrow V_h$ whose local form is defined by
\begin{equation}
\label{def-Mh-os}
\mathcal{M}_h|_K \big(\sum_j c_j \bar{\psi}_{\chi_j}^K\big) =\sum_j c_j \chi_j^K.
\end{equation}
Then we define finite element spaces for the oversampling expanded
mixed MsFEM as follows:
\begin{eqnarray*}
\begin{split}
X_{1,h}^{os}&:=\{\theta \in [L^2(\Omega)]^d:  \theta|_K\in span\{\bar{\eta}_{\chi}^K\} \quad \text{for each $K\in \mathfrak{T}_h$} \}\\
X_{2,h}^{os}&:=\{u_h \in \tilde{X}_{2,h}^{os}:  \mathcal{M}_h u_h \in V_h \subset H(div, \Omega) \}.
\end{split}
\end{eqnarray*}
Here, $\mathcal{M}_h u_h \in V_h$ is to impose some intrinsic
continuity of the normal components of velocity multiscale basis
functions across the interfaces $\mathcal{F}_h$.  However,
$X_{2,h}^{os}\not\subset H(div, \Omega)$ and so the oversampling
multiscale method is nonconforming.  The expanded oversampling mixed
MsFEM formulation of the global problem (\ref{exp-fem-cont}) reads: find
$(\theta_h^{os}, u_h^{os}, p_h^{os})\in X_{1,h}^{os}\times X_{2,h}^{os} \times Q_h$
such that
\begin{eqnarray}
\label{exp-fem-num-os}
\begin{cases}
\begin{split}
(k\theta_h^{os}, \xi_h)+(u_h^{os}, \xi_h)&=0   \quad \forall \xi_h\in X_{1,h}^{os}\\
(\theta_h^{os}, \tau_h)+\sum_K (p_h^{os}, div\tau_h)_K &=0  \quad  \forall \tau_h\in X_{2,h}^{os}\\
\sum_K (div u_h^{os}, q_h)_K &=(f, q_h)  \quad q_h\in Q_h
\end{split}
\end{cases}
\end{eqnarray}
subject to the global boundary conditions.

We define operator
$C_h^{os}: \tilde{X}_{2,h}^{os}\longrightarrow (\Pi_h^l)'$ by
\[
[C_h^{os}(\tau_h), \pi_h]:=\sum_K (\mathcal{M}_h(\tau_h)\cdot n, \pi_h)_{\pK},  \quad \forall \tau_h \in \tilde{X}_{2,h}^{os}, \pi_h\in \Pi_h^l.
\]
Then the hybridization of (\ref{exp-fem-num-os}) is to find
$(\theta_h^{os}, u_h^{os}, p_h^{os},\lambda_h^{os}) \in
X_{1,h}^{os}\times \tilde{X}_{2,h}^{os} \times Q_h \times \Pi_h^l$
such that
\begin{eqnarray}
\label{exp-fem-hybrid-os}
\begin{cases}
\begin{split}
(k \theta_h^{os}, \xi_h)+(u_h^{os}, \xi_h)&=0   \quad \forall \xi_h\in X_{1,h}^{os}\\
(\theta_h^{os}, \tau_h)+\sum_K (p_h^{os}, div\tau_h)_K &=[C_h^{os}(\tau_h), \lambda_h^{os}] \quad  \forall \tau_h\in \tilde{X}_{2,h}^{os}\\
\sum_K (div u_h^{os}, q_h)_K &=(f, q_h)  \quad q_h\in Q_h\\
[C_h^{os}(u_h), \pi_h] &=0  \quad \pi_h \in \Pi_h^l.
\end{split}
\end{cases}
\end{eqnarray}


\subsection{Global expanded mixed MsFEM}
\label{global-ExMMsFEM}

To remove the resonance error and substantially improve the accuracy, we can
use global information to construct the multiscale basis functions.
Suppose that there exist global fields $u_1, \cdots, u_N$ that
capture the non-local features of the solution of equation
(\ref{target-elliptic}).  Following the global mixed MsFEM framework
proposed in \cite{aej08}, we define the multiscale basis equations for
the global expanded mixed MsFEM by
\begin{eqnarray}
\label{global-basis}
\begin{cases}
\begin{split}
\psi_{i, \chi}^K+k \eta_{i, \chi}^K &=0  \quad \text{in} \quad K\\
\eta_{i, \chi}^K -\nabla \phi_{i, \chi}^K &=0 \quad \text{in} \quad K\\
div(\psi_{i, \chi}^K) &= div \chi^K  \quad \text{in} \quad K\\
\psi_{i, \chi}^K\cdot n_{\pK} &=b_i^{\pK}   \quad \text{on} \quad \pK,
\end{split}
 \end{cases}
 \end{eqnarray}
where
$
  b_i^{\pK}:=\displaystyle{
  \frac{\int_{\pK} \chi^K\cdot  n_{\pK}ds}{\int_{\pK}u_i\cdot n_{\pK}ds}
  u_i\cdot n_{\pK}}
$
and $i=1,..., N$ are indices of the global information. References
\cite{aej08, bo09,oz07} provide some options for the global fields. To
reduce the computational cost of using global fields, we can pre-compute
these fields $u_1,\cdots, u_N$ at some intermediate scale \cite{jem10}.

The finite element spaces for the global expanded mixed MsFEM are
defined by
\begin{eqnarray*}
\begin{split}
X_{1,h}^g&:=\{\theta \in [L^2(\Omega)]^d:  \theta|_K\in span\{\eta_{i,\chi}^K\} \quad \text{for each $K\in \mathfrak{T}_h$} \}\\
X_{2,h}^g&:=\{u \in H(div, \Omega):  u|_K\in span\{\psi_{i,\chi}^K\} \quad \text{for each $K\in \mathfrak{T}_h$} \}\\
X_h^g&:=X_{1,h}^g\times X_{2,h}^g.
\end{split}
\end{eqnarray*}
Consequently the global expanded mixed MsFEM formulation of
the global problem (\ref{exp-fem-cont}) reads: find $\big((\theta_h^g, u_h^g),
p_h^g\big)\in X_{h}^g \times Q_h$ such that
\begin{eqnarray}
\label{exp-fem-num-global}
\begin{cases}
\begin{split}
(k \theta_h^g, \xi_h)+(u_h^g, \xi_h)&=0   \quad \forall \xi_h\in X_{1,h}^g\\
(\theta_h^g, \tau_h)+(p_h^g, div\tau_h)&=0  \quad  \forall \tau_h\in X_{2,h}^g\\
(div u_h^g, q_h)&=(f, q_h)  \quad q\in Q_h
\end{split}
\end{cases}
\end{eqnarray}
subject to the global boundary conditions.

We define the MsFEM velocity space $\tilde{X}_{2,h}^g$ and the finite
element space $\Pi_h^g$ of the Lagrange multipliers for the hybrid
formulation of (\ref{exp-fem-num-global}) by
\begin{eqnarray*}
\begin{split}
\tilde{X}_{2,h}^g&:=\{u \in [L^2(\Omega)]^d:  u|_K\in span\{\psi_{i,\chi}^K\} \quad \text{for each $K\in \mathfrak{T}_h$} \}\\
\Pi_h^g&:=\big\{\pi\in L^2(\mathcal{F}_h): \pi|_{e}\in X_{2,h}^g\cdot n_{e}  \quad \text{for each  $e \in \mathcal{F}_h$},  \quad \pi|_e=0 \quad \text{for $e\subset \partial \Omega$}  \big\}.
\end{split}
\end{eqnarray*}
The functions in $\tilde{X}_{2,h}^g$ may not have normal
continuity. The relation between $X_{2,h}^g$ and $\tilde{X}_{2,h}^g$
is expressed by the identity $X_{2,h}^g=\tilde{X}_{2,h}^g\cap H(div,
\Omega)$.

Let the operator $C_h^g: \tilde{X}_{2,h}^g\longrightarrow (\Pi_h^g)'$
be defined in a similar way to (\ref{def-C}).  Then the hybridization
of the global expanded mixed MsFEM formulation of
(\ref{exp-fem-num-global}) reads: find
$(\bar{\theta}_h^g,\bar{u}_h^g, \bar{p}_h^g, \lambda_h^g )
\in X_{h}^g \times \tilde{X}_{2,h}^g \times Q_h\times \Pi_h^g$
such that
\begin{eqnarray}
\label{exp-fem-hybrid-global}
\begin{cases}
\begin{split}
(k \bar{\theta}_h^g, \xi_h)+(\bar{u}_h^g, \xi_h)&=0   \quad \forall \xi_h\in X_{1,h}^g\\
(\bar{\theta}_h^g, \tau_h)+\sum_K (\bar{p}_h^g, div\tau_h)&=[C_h^g(\tau_h), \lambda_h^g]  \quad  \forall \tau_h\in \tilde{X}_{2,h}^g\\
\sum_K (div \bar{u}_h^g, q_h)&=(f, q_h)  \quad q\in Q_h\\
[C_h^l(\bar{u}_h^g), \pi_h]&=0   \quad \pi_h\in \Pi_h^g.
\end{split}
\end{cases}
\end{eqnarray}
We note that (\ref{exp-fem-hybrid-local}), (\ref{exp-fem-hybrid-os})
and (\ref{exp-fem-hybrid-global}) can be rewritten as the operator
formulation
\begin{equation}
\label{operator-hybrid}
\begin{bmatrix}
 A_{1,h} & B_{1,h}^* & O & O \\
B_{1,h} & O &  B_h^* & -C_h^* \\
O & B_h & O & O\\
O & -C_h&  O & O &
\end{bmatrix}
 \left[ \begin{array}{c} \bar{\theta}_h \\ \bar{u}_h\\ \bar{p}_h \\ \lambda_h \end{array} \right]
=\left[\begin{array}{c} O\\ O\\ F_h \\ O \end{array}\right].
\end{equation}
This is actually the hybrid approximation of the operator equation
(\ref{operator-num-eq}).


\section{Convergence analysis}

We will focus on the analysis of the local expanded mixed MsFEM.  The
analysis of the oversampling expanded mixed MsFEM and the global
expanded mixed MsFEM is very similar to the local expanded mixed
MsFEM but requires additional notation and definitions.  We sketch the
analysis of the oversampling expanded mixed MsFEM and global expanded
mixed MsFEM and present their convergence results.  For the local
expanded mixed MsFEM, we consider micro-scales that can be treated
in terms of periodic homogenization and $G-$convergent homogenization.

For the analysis, we will focus on the case when the velocity
space on $K$, $V_h(K)$, is the lowest-order Raviart-Thomas space ($RT_0$)
and $\chi^K\in V_h(K)$.
We use $\chi^K$ to build the source term and boundary conditions for the MsFEM
basis defined in equation~(\ref{local-basis}).

\subsection{Inf-sup condition for expanded mixed MsFEMs}

In this subsection, we discuss the inf-sup condition associated with
problem (\ref{exp-fem-num-local}).  To simplify the presentation,
we define three discrete operators associated with
(\ref{exp-fem-num-local}). Specifically, let $A_{1,h}:X_{1,h}^l\longrightarrow
X_{1,h}'$, $B_{1,h}: X_{1,h}^l\longrightarrow X'_{2,h}$ and $B_h:
X_{2,h}^l\longrightarrow Q'_h$, then the operators are defined as
\begin{eqnarray*}
\begin{split}
[A_{1,h} (\theta), \xi] :&=(k\theta, \xi),  \quad \forall \xi \in  X_{1,h}^l\\
[B_{1,h}(\theta), v]&:=(\theta, v), \quad \forall v\in X_{2,h}^l\\
[B_h(u), q]&:=(div (u), q),  \quad \forall q\in Q_h.
\end{split}
\end{eqnarray*}
By Lemma \ref{inf-sup-abstract-lem}, we require that the operator
$A_{1,h}$ is bounded and positive for well-posedness of problem
(\ref{exp-fem-num-local}).  If $k$ satisfies the assumption
\begin{equation}
\label{k-assum}
k_{\min} |\zeta|^2\leq \zeta^T k (x)\zeta \leq k_{\max}|\zeta|^2 \quad \text {for $\forall \zeta\in \mathbb{R}^d$ and $\forall x\in \Omega $},
\end{equation}
then the operator $A_{1,h}$ is bounded and positive. However, we can
weaken the assumption (\ref{k-assum}); for example, $k(x)$ may be
partially vanishing inside fine cells and positive and bounded
everywhere else. Then the operator $A_{1,h}$ is still positive and
bounded.  For simplicity of presentation, we use assumption
(\ref{k-assum}) for the analysis.

We have the following lemma for the discrete inf-sup condition of
problem (\ref{exp-fem-num-local}).
\begin{lemma}
\label{inf-sup}
For any $q_h\in Q_h$, there exists a positive constant $C_1$
independent of $h$ such that
\begin{equation}
\label{sup-B}
\sup_{v_h \in X_{2,h}^l\setminus \{0\}} \frac{[B_h(v_h), q_h]}{\|v_h\|_{X_2}}\geq C_1 \|q_h\|_{Q}.
\end{equation}
Let $\ker(B_h)$ be the kernel of the operator $B_h$. Then for any $v_h
\in \ker(B_h)$, there exists a positive constant $C_2$ independent of
$h$ such that
\begin{equation}
\label{sup-B1}
\sup_{\xi_h \in X_{1,h}^l\setminus \{0\}} \frac{[B_{1,h}(\xi_h), v_h]}{\|\xi_h\|_{X_1}}\geq C_2 \|v_h\|_{X_2}.
\end{equation}
\end{lemma}
\begin{proof}
We define $v_h^*\in V_h$, whose localization on $K$ is
$v_h^*|_K =\sum_j a_j^K \chi^K_j\in V_h(K)$, where $j$ is the
index of a face of $K$.  Consider a map $M$ whose local form
is defined by
  $M_K v_h^*=\sum_j a_j^K \psi_{\chi_j}^K \in X_{2,h}^l(K)$.
Then $M: V_h\longrightarrow X_{2,h}^l$ is a one-to-one map.
It is easy to check for any $v_h^*\in V_h$,
\begin{equation}
\label{inf-sup-eq1}
div(Mv_h^*)=div(v_h^*) \ \ \text{in}  \ \ K.
\end{equation}
Let $z^K=\sum_j a_j^K \phi_{\chi_j}^K$. Then $M_K v_h^*=-k \nabla z^K$
and we have
\begin{eqnarray}
\begin{split}
&\|M_K v_h^*\|_{0,K}^2 = \int_K k \nabla z^K \cdot k \nabla z^Kdx \leq C \int_K k \nabla z^K \cdot  \nabla z^Kdx\\
&= -C\int_K M_K v_h^*\cdot \nabla z^Kdx=C(\int_K div(M_K v_h^*)z^Kdx-\int_{\pK} (M_K v_h^*)\cdot n z^Kds)\\
&=C(\int_K div(v_h^*)z^Kdx-\int_{\partial K}  v_h^*\cdot n z^Kds)=-C\int_K v_h^*\cdot \nabla z^Kdx\\
&=C\int_K  v_h^* \cdot k^{-1}M_K v_h^*dx \leq C \|v_h^*\|_{0,K}\|M_K v_h^*\|_{0,K}.
\end{split}
\end{eqnarray}
This gives $\|M_K v_h^*\|_{0,K}\leq C \|v_h^*\|_{0,K}$. Consequently,
combining with (\ref{inf-sup-eq1}) implies that for any
$v_h^*\in V_h^*$,
\begin{equation}
\label{inf-sup-eq2}
\|M v_h^*\|_{X_2}\leq C \|v_h^*\|_{X_2}.
\end{equation}
Consequently, since the inf-sup condition holds for the classic mixed
FE pair $(V_h, Q_h)$, it follows that for any $q_h\in Q_h$
\begin{eqnarray}
\begin{split}
\displaystyle{\sup_{v_h\in X_{2,h}^l\setminus \{0\}}\frac{ (div v_h,  q_h) }{\|v_h\|_{X_2}}}
&\geq \displaystyle{\sup_{v_h^*\in V_h\setminus \{0\}}\frac{ (div (M v_h^*), q_h) }{\|M v_h^*\|_{X_2}}}\\
&\geq {1\over C}  \displaystyle{\sup_{v_h^*\in V_h\setminus \{0\}}\frac{( div v_h^*, q_h) }{\| v_h^*\|_{X_2}}} \geq {C^*\over C}\|q_h\|_{Q},
\end{split}
\end{eqnarray}
where we have used (\ref{inf-sup-eq2}) in the second step.  The proof
of (\ref{sup-B}) is complete.

It is obvious that by definition
 \[
 \ker(B_h)=\{v_h\in X_{2,h}^l:  div (v_h)=0\},
 \]
and hence $\ker(B_h)\subset X_{2,h}^l \subset X_{1,h}^l$.  Then for any $v_h\in \ker(B_h)$,
 \[
 \sup_{\xi_h \in X_{1,h}^l\setminus \{0\}} \frac{[B_{1,h}(\xi_h), v_h]}{\|\xi_h\|_{X_1}}= \sup_{\xi_h \in X_{1,h}^l\setminus \{0\}} \frac{(\xi_h, v_h)}{\|\xi_h\|_{X_1}}=\|v_h\|_{0, \Omega}
 =\|v_h\|_{X_2},
 \]
which completes the proof of (\ref{sup-B1}).
\end{proof}

Using Lemma \ref{inf-sup-abstract-lem}, Lemma \ref{inf-sup} and Lemma
\ref{strang-lem} gives the following theorem.
\begin{theorem}
\label{inf-sup-thm}
The problem (\ref{exp-fem-num-local}) has a unique solution $\big((\theta_h^l, u_h^l), p_h^l\big)\in X_{h}^l\times Q_h$.
Let $\big((\theta, u), p\big)$ be the solution of (\ref{exp-fem-cont}). Then there exists a positive constant $C$
independent of $h$  such that
\begin{equation}
\label{cea-est1}
\|\big((\theta, u), p\big)-\big((\theta_h^l, u_h^l), p_h^l\big)\|_{X\times Q}\leq C\inf_{ \big( (\xi_h, v_h), q_h)\big)\in X_{h}^l\times Q_h} \|\big((\theta, u), p\big)-\big( (\xi_h, v_h), q_h)\big)\|_{X\times Q}.
\end{equation}
\end{theorem}

Following the proof of Lemma \ref{inf-sup} and using the argument of
Lemma 4.4 in \cite{ch03}, we can obtain the inf-sup condition for the
oversampling expanded mixed MsFEM problem (\ref{exp-fem-num-os}).
Then applying Lemma \ref{strang-lem} to (\ref{exp-fem-num-os}), we
have the estimate
\begin{eqnarray}
\label{os-ineq0}
\begin{split}
&\sum_K \|\big((\theta, u), p \big)-\big((\theta_h^{os}, u_h^{os}), p_h^{os}\big)\|_{X(K)\times Q(K)}\\
&\leq C\big\{ \inf_{(\xi_h, v_h, q_h)\in X_{1,h}^{os}\times X_{2,h}^{os} \times Q_h} \sum_{K} \|((\theta, u), p)-((\xi_h, v_h), q_h)\|_{X(K)\times Q(K)}  \big\}\\
& +\sum_{v_h\in X_{2,h}^{os}\setminus \{0\}} \frac{-(\theta, v_h)-\sum_{K} (div(v_h), p)}{\sum_K \|v_h\|_{div, K}}.
\end{split}
\end{eqnarray}
The last term of (\ref{os-ineq0}) is the consistency error caused by
the oversampling velocity space.

Under suitable conditions described in \cite{aej08}, we can similarly
obtain the inf-sup conditions for the global expanded mixed multiscale finite element
problem (\ref{exp-fem-num-global}) and the C\'{e}a-type estimate,
\begin{equation}
\|\big((\theta, u), p\big)-\big((\theta_h^g, u_h^g), p_h^g\big)\|_{X\times Q}\leq C\inf_{ \big( (\xi_h, v_h), q_h)\big)\in X_{h}^g\times Q_h} \|\big((\theta, u), p\big)-\big( (\xi_h, v_h), q_h)\big)\|_{X\times Q}.
\end{equation}


\subsection{Equivalence between expanded mixed MsFEM and its hybridization}

In this subsection, we show that expanded mixed MsFEM formulation and
its hybridization produce the same vector functions (i.e., the gradient
of pressure and velocity).  The result will be helpful to
proceed with the convergence analysis for expanded mixed MsFEM and its
hybrid formulation.  For simplicity, we focus on the case of local
mixed MsFEM.

We first have the following lemma for operator $C_h^l$.
\begin{lemma}
\label{kerC-lem}
Let $C_h^l$ be defined in (\ref{def-C}).  Then
\[
\ker C_h^l=X_{2,h}^l,  \quad  \ker (C_h^l)^*=\{0\}.
\]
\end{lemma}

\begin{proof}
By definition of $X_{2,h}^l$ and (\ref{def-C}), it immediately follows that
\[
\ker C_h^l=X_{2,h}^l.
\]
Let $K^*$ be any cell in $\mathfrak{T}_h$ and $e^*$ be any face of $K^*$. Let
$\tau^*\in \tilde{X}_{2,h}^l$ be such that
\[
\tau_h^*|_K=0 \quad \forall K\neq K^*
\]
and defined on $K^*$ by
\begin{eqnarray*}
\begin{cases}
\tau_h^*\cdot n_e&=0  \quad \text{for}  \quad e\neq e^*\\
\tau_h^*\cdot n_{e^*} &= \chi_{e^*}^{K^*}\cdot n_{e^*},
\end{cases}
\end{eqnarray*}
where $\chi_{e^*}^{K^*}\in V_h(K^*)$ is any basis function associated
with face $e^*$.  Then $[(C_h^l)^*(\pi_h), \tau_h^*]=0$ means that
$(\chi_{e^*}^{K^*} \cdot n, \pi_h)_{e^*}=0$ for any basis function
$\chi_{e^*}^{K^*}\in V_h(K^*)$ associated with face $e^*$.  The
classic mixed FEM theory implies that $\pi_h=0$ on $e^*$.  Since $K^*$
is any cell in $\mathfrak{T}_h$ and $e^*$ is an arbitrary interface,
$\pi_h=0$ on $\mathcal{F}_h$.  This completes the proof.

\end{proof}

Next we show that problem $(\ref{exp-fem-num-local})$ and problem
(\ref{exp-fem-hybrid-local}) produce the same solution.

\begin{theorem}
\label{equi-thm-local}
Problem $(\ref{exp-fem-num-local})$ has a unique solution
$\big((\theta_h^l, u_h^l), p_h^l\big)\in X_h^l\times Q_h$.  Problem
(\ref{exp-fem-hybrid-local}) has a unique solution
$(\bar{\theta}_h^l, \bar{u}_h^l, \bar{p}_h^l, \lambda_h^l )
\in X_{1,h}^l\times \tilde{X}_{2,h}^l \times Q_h\times \Pi_h^l$.
Moreover,
$ \big((\bar{\theta}_h^l, \bar{u}_h^l),\bar{p}_h^l\big)
= \big((\theta_h^l, u_h^l), p_h^l\big) $.

\end{theorem}

\begin{proof}

Following the proof of Theorem 3.1 in \cite{awy97} and using a
standard duality argument, we obtain that
\[
  \|\theta_h^l\|_{0, \Omega}+\|u_h^l\|_{0, \Omega}+\|p_h^l\|_{0, \Omega}
  \leq C\|f\|_{0, \Omega},
\]
which shows (\ref{exp-fem-num-local}) has a unique solution by setting
$f=0$.

Let $\big((\theta_h^l, u_h^l), p_h^l\big)$ be the unique solution of
$(\ref{exp-fem-num-local})$. We define a linear operator by
\begin{equation}
\label{def-L-1}
L(\tau_h)=(\theta_h^l, \tau_h)+\sum_K (p_h^l, div \tau_h)_K,  \quad \forall \tau_h\in \tilde{X}_{2,h}^l.
\end{equation}
It is clear that $L(\tau_h)=0$ for all $\tau_h\in X_{2,h}^l$, i.e, $X_{2,h}^l \subset \ker L$ and
\[
L\in [\ker(C_h^l)]^{\perp}:=\{ l\in (\tilde{X}_{2,h}^l)':  <l, \tau_h>=0, \quad \forall \tau_h\in \ker(C_h^l)\}.
\]
By the Closed Range Theorem, it follows that
\[
[\ker(C_h^l)]^{\perp}=\Re( (C_{h}^l)^*).
\]
This means that there exists a $\lambda_h\in \Pi_h^l$ such that
$(C_h^l)^*(\lambda_h)=L$, which gives
\begin{equation}
\label{def-L-2}
L(\tau_h)=[C_h^l(\tau_h), \lambda_h^l] \quad \forall \tau_h\in \tilde{X}_{2,h}^l.
\end{equation}
Then by Lemma \ref{kerC-lem}, $\lambda_h^l$ is unique.  Combining
(\ref{def-L-1}) and (\ref{def-L-2}) gives
\begin{equation}
\label{def-L-3}
(\theta_h^l, \tau_h)+\sum_K (p_h^l, div \tau_h)_K=[C_h^l(\tau_h), \lambda_h^l] \quad \forall \tau_h\in \tilde{X}_{2,h}^l.
\end{equation}
The equation (\ref{def-L-3}) is exactly the same as the second
equation in (\ref{exp-fem-hybrid-local}).  The fourth equation in
(\ref{exp-fem-hybrid-local}) and Lemma \ref{kerC-lem} show that
$\bar{u}_h\in X_{2,h}^l$.  By uniqueness of solutions of
(\ref{exp-fem-num-local}) and uniqueness of $\lambda_h^l$, we
immediately obtain that $\big((\bar{\theta}_h^l, \bar{u}_h^l),
\bar{p}_h^l\big)$ is unique and
\[
\big((\bar{\theta}_h^l, \bar{u}_h^l), \bar{p}_h^l\big)=\big((\theta_h^l, u_h^l), p_h^l\big).
\]

\end{proof}

By a procedure similar to Theorem \ref{equi-thm-local}, we have the
following theorems for oversampling expanded mixed MsFEM and global
expanded mixed MsFEM, respectively.
\begin{theorem}
\label{equi-thm-os}
Problem (\ref{exp-fem-hybrid-os}) has a unique solution $(\bar{\theta}_h^{os}, \bar{u}_h^{os}, \bar{p}_h^{os}, \lambda_h^{os} )\in X_{1,h}^{os} \times \tilde{X}_{2,h}^{os} \times Q_h\times \Pi_h^l$.
Moreover, $\big((\bar{\theta}_h^{os}, \bar{u}_h^{os}), \bar{p}_h^{os}\big)=\big((\theta_h^{os}, u_h^{os}), p_h^{os}\big)$, where $\big((\theta_h^{os}, u_h^{os}), p_h^{os}\big)$ is the unique solution of problem (\ref{exp-fem-num-os}).
\end{theorem}

\begin{theorem}
\label{equi-thm-global}
Problem (\ref{exp-fem-hybrid-global}) has a unique solution $(\bar{\theta}_h^g, \bar{u}_h^g, \bar{p}_h^g, \lambda_h^g )\in X_{1,h}^g \times \tilde{X}_{2,h}^g \times Q_h\times \Pi_h^g$.
Moreover, $\big((\bar{\theta}_h^g, \bar{u}_h^g), \bar{p}_h^g\big)=\big((\theta_h^g, u_h^g), p_h^g\big)$, where $\big((\theta_h^g, u_h^g), p_h^g\big)$ is the unique solution of $(\ref{exp-fem-num-global})$.
\end{theorem}

Theorem \ref{equi-thm-local},  Theorem \ref{equi-thm-os} and Theorem \ref{equi-thm-global} allow us to identify $\big((\bar{\theta}_h^l, \bar{u}_h^l), \bar{p}_h^l\big)$ with $\big((\theta_h^l, u_h^l), p_h^l\big)$,
 $\big((\bar{\theta}_h^{os}, \bar{u}_h^{os}), \bar{p}_h^{os}\big)$ with $\big((\theta_h^{os}, u_h^{os}), p_h^{os}\big)$, and
     $\big((\bar{\theta}_h^g, \bar{u}_h^g), \bar{p}_h^g\big)$ with $\big((\theta_h^g, u_h^g), p_h^g\big)$, respectively. So we can drop the upper bars in corresponding
hybrid systems (\ref{exp-fem-hybrid-local}),  (\ref{exp-fem-hybrid-os}), and (\ref{exp-fem-hybrid-global}), respectively. We note that
$\bar{u}^l_h=u_h^l$ is an identity  in the sense of vector-valued functions. However, the corresponding coefficient arrays on
their numerical representations are not equal to each other, because $dim(X_{2,h}^l)\neq dim(\tilde{X}_{2,h}^l)$.  The same
relationships hold for the oversampling expanded mixed MsFEM and the global expanded mixed MsFEM as well.
Thus, to simplify the notation, we will drop the upper bars from
(\ref{exp-fem-hybrid-local}), (\ref{exp-fem-hybrid-os})  and (\ref{exp-fem-hybrid-global}) in the rest of the paper.


\subsection{Error analysis for the periodic case}
\label{peroidc-case}

In this subsection, we consider the case $k(x)=k({x\over \epsilon})$,
where the parameter $\epsilon$ characterizes small scales representing
small physical lengths. To highlight the dependence on the small
parameter $\epsilon$, we rewrite the solution of (\ref{exp-fem-cont})
as $\big((\se, \ue), \pe\big)$.  We present an error analysis when the
$\epsilon$ scales are periodic.  The local expanded mixed MsFEM and
the oversampling expanded mixed MsFEM will be considered in the
following subsections.

We briefly recall the relevant homogenization results for the periodic case.
Specifically, let
$y={x\over \epsilon}$ and $k(y)$ be a periodic function in the unit
cube $Y=[0,1]^d$.  In this case, we can compute $k^*$ in the following
way.  Let $\mathcal{N}=\{\mathcal{N}_1,\cdots, \mathcal{N}_d\}$ solve
the auxiliary equations
\begin{eqnarray}
\begin{split}
\begin{cases}
-div_y (k(y)\nabla \mathcal{N}_i) &=div_y(k(y) e_i ) \ \ \text{in} \ \ Y\\
\langle  \mathcal{N}_i(y)\rangle _Y&=0.
\end{cases}
\end{split}
\end{eqnarray}
Here $e_i$ ($i=1,\cdots, d$) is the unit vector in
$\mathbb{R}^d$. Then the homogenized tensor $k^*$ is defined as
(see, e.g., \cite{jko94})
\[
k^*=\langle k(\nabla \mathcal{N} +I)\rangle_Y.
\]
Let $p^*$ solve   the homogenized equation of (\ref{target-elliptic})  by
\begin{eqnarray}
\label{homog-elliptic}
\begin{cases}
\begin{split}
 -div(  k^*\nabla p^*)&=f(x) \quad \text{in} \quad\Omega\\
 p^* &=0 \quad \text{on} \quad
\partial\Omega.
\end{split}
\end{cases}
\end{eqnarray}
We define the homogenized velocity $u^*$ by $u^*:=-k^*\nabla p^*$.

Since we consider  $\chi^K\in V_h(K):=RT_0(K)$, we have $div\chi^K={1\over |K|}$ and $\chi_e^K\cdot n_e ={1\over |e|}$ on $e$ and 0 otherwise.
For this case, $Q_h$ is a piecewise constant space.

\subsubsection{Analysis of local expanded  MsFEM}

Let $\mathcal{I}_h^*: [H^1(\Omega)]^d\longrightarrow V_h$ be the interpolation operator of $RT_0$ for velocity.
Let $\mathcal{P}_h: L^2(\Omega)\longrightarrow Q_h$ be an orthogonal $L^2$ projection.
We define the multiscale interpolation operator  $\mathcal{I}_h: [H^1(\Omega)]^d\longrightarrow \tilde{X}_{2,h}^l$ for velocity by
\[
\mathcal{I}_h(v):=\sum_{e\in \mathcal{F}_h} (\int_e v\cdot n_{e} ds) \psi_{\chi_e},
\]
where $\psi_{\chi_e}$ is a multiscale velocity basis function associated with interface $e$.
We define the $L^2$ interface projection  $\mathcal{P}_{\partial}: L^2({\mathcal{F}_h})\longrightarrow \Pi_h^l$ such
that
\[
(\mathcal{P}_{\partial} p-p, \pi_h)_e=0  \quad \forall \pi_h\in \Pi_h^l \quad \text{ and} \quad \forall e\in\mathcal{F}_h.
\]
This interface projection implies that $\mathcal{P}_{\partial} p|_e=(\int_e pds)\chi_e\cdot n_e$.
We define a discrete norm on $L^2({\mathcal{F}_h})$ by
\[
\|\pi\|_{-{1\over 2}, h}^2=\sum_{K\in \mathfrak{T}}\sum_{e\subset \pK} h_K \|\pi\|_{0,e}^2.
\]

\begin{lemma}
If $\chi_e^K$ is a lowest Raviart-Thomas element, then the
corresponding multiscale basis function $\psi_{\chi_e}^K$ satisfies
\begin{equation}
\label{basis-est}
\|\psi_{\chi_e}^K\|_{0,K} \leq C |e|^{-{1\over 2}} h_K^{1\over 2}.
\end{equation}
\end{lemma}

\begin{proof}
By basis equation (\ref{local-basis}),  it follows that
\begin{eqnarray*}
\begin{cases}
\begin{split}
-div (k\nabla \phi_{\chi_e}^K)&={1\over |K|} \quad \text{in K}\\
-k\nabla \phi_{\chi_e}^K\cdot n_e& ={1\over |e|} \quad \text{on $e$}\\
-k\nabla \phi_{\chi_e}^K\cdot n& =0 \quad \text{on $\partial K\setminus e$}\\
{1\over |K|}\int_K \phi_{\chi_e}^K &=0.
\end{split}
\end{cases}
\end{eqnarray*}
Multiplying $\phi_{\chi_e}^K$ and using integration by parts for the above equation, we have
\[
(k\nabla \phi_{\chi_e}^K, \nabla \phi_{\chi_e}^K)-{1\over |e|}\int_e \phi_{\chi_e}^K=0.
\]
The Cauchy-Schwarz inequality implies that
\begin{equation}
\label{ineq1-lem4.7}
(k\nabla \phi_{\chi_e}^K, \nabla \phi_{\chi_e}^K)={1\over |e|}\int_e \phi_{\chi_e}^K\leq |e|^{-{1\over 2}}\|\phi_{\chi_e}^K\|_{0, e}.
\end{equation}
Applying the trace theorem and scaling argument \cite{BreFort91}, it follows that
\begin{equation}
\label{ineq2-lem4.7}
|e|^{-{1\over 2}} \|\phi_{\chi_e}^K\|_{0, e}\leq C |e|^{-{1\over 2}} h_K^{1\over 2} \|\nabla \phi_{\chi_e}^K\|_{0,K}\leq C |e|^{-{1\over 2}} h_K^{1\over 2} \| \psi_{\chi_e}^K\|_{0,K},
\end{equation}
Therefore, by (\ref{ineq1-lem4.7}) and (\ref{ineq2-lem4.7}) we have
\[
\| \psi_{\chi_e}^K\|_{0, K}^2\leq C(k\nabla \phi_{\chi_e}^K, \nabla \phi_{\chi_e}^K)\leq C |e|^{-{1\over 2}} h_K^{1\over 2} \| \psi_{\chi_e}^K\|_{0,K}.
\]
This completes the proof.
\end{proof}

 We define the homogenization equation  of the basis  equation (\ref{local-basis}) by
\begin{eqnarray}
\label{local-basis-hom}
\begin{cases}
\begin{split}
\psi_{\chi}^{*,K}+k^*\eta_{\chi}^{*,K}&=0\\
\eta_{\chi}^{*,K}-\nabla \phi^{*,K} &=0\\
div(\psi_{\chi}^{*,K}) &= div \chi^K  \quad \text{in} \quad K\\
\psi_{\chi}^{*,K}\cdot n_{\pK} &=\chi^K \cdot n_{\pK}   \quad \text{on} \quad \pK  \, .
\end{split}
\end{cases}
\end{eqnarray}
By Theorem \ref{exp-elliptic-thm} and the uniqueness of solution of (\ref{local-basis-hom}), we have
\begin{equation}
\label{hom-base-RT}
\psi_{\chi}^{*,K}=\chi^K,  \quad \eta_{\chi}^{*,K}=(k^*)^{-1} \chi^{K}.
\end{equation}

Let $\we(x)$  be the solution (up to a constant)  of the equation
\begin{eqnarray}
\label{w-eq}
\begin{split}
\begin{cases}
-div (\ke(x)\nabla \we) &=div (\mathcal{I}_h^* u^*|_K) \ \  \text{in} \ \ K\\
-\ke(x)\nabla \we \cdot n_{\pK} &=(\mathcal{I}_h^* u^*|_K)\cdot n_{\pK} \ \ \text{on} \ \ \partial K.
\end{cases}
\end{split}
\end{eqnarray}
Then the homogenized equation of (\ref{w-eq}) is
\begin{eqnarray}
\label{wh-eq}
\begin{split}
\begin{cases}
-div (k^*\nabla \wh) &=div (\mathcal{I}_h^* u^*|_K) \ \  \text{in} \ \ K\\
-k^*\nabla \wh \cdot n_{\pK} &=(\mathcal{I}_h^* u^*|_K) \cdot n_{\pK} \ \ \text{on} \ \ \partial K.
\end{cases}
\end{split}
\end{eqnarray}

By a straightforward calculation (ref. \cite{jem10}), we have the following lemma.
\begin{lemma}
\label{vh-eq}
Let $\we$ and $\wh$ be defined in (\ref{w-eq}) and (\ref{wh-eq}), respectively. Then
\begin{eqnarray}
\begin{split}
-\ke\nabla \we&=\mathcal{I}_h u^*|_K \ \ \text{in}  \ \ K \\
-k^*\nabla \wh &=\mathcal{I}^*_h u^*|_K  \ \ \text{in}  \ \ K.
\end{split}
\end{eqnarray}
\end{lemma}

Moreover, we have
\begin{lemma}
\label{gradw-lem}
Let $\we$ and $\wh$ be defined in (\ref{w-eq}) and (\ref{wh-eq}), respectively. Then
\[
\nabla \we\in X_{1,h}^l(K), \quad \nabla \wh\in V_h(K).
\]
\end{lemma}
\begin{proof}
It is obvious that
\[
-div (\ke(x)\nabla \we) =div (\mathcal{I}_h^* u^*|_K)=\big( \sum_{e\subset \pK} \int_{e} u^*\cdot n ds (div \chi_e^K)\big),
\]
and
\[
-\ke(x)\nabla \we \cdot n_e =(\mathcal{I}_h^*|_K) u^*\cdot n_e=(\sum_{e_j\subset \pK} \int_{e_j} u^*\cdot n ds \chi_{e_j}^K)\cdot n_e=(\int_{e} u^*\cdot n ds)\chi_{e}^K\cdot n_e.
\]
By Theorem \ref{exp-elliptic-thm} and equation (\ref{local-basis}), it follows
\[
\nabla \we=\sum_{e\subset \pK} \big((\int_{e} u^*\cdot n ds) \eta_{\chi_e}^K\big).
\]
Similarly, we can show that by (\ref{hom-base-RT})
\[
\nabla \wh=(k^*)^{-1} \sum_{e\subset \pK} \big((\int_{e} u^*\cdot n ds) \chi_e^K\big).
\]
This completes the proof.
\end{proof}

 For the relationship between $\pe$ and $\we$,  we have the following lemma.
\begin{lemma}
\label{p-w-lem}
Let  $\we$ be defined in (\ref{w-eq}). Then
\begin{eqnarray}
|\pe-\we|_{1, K}\leq C (\epsilon+h+\sqrt{\epsilon h})\|p^*\|_{2, K} +C\sqrt{\epsilon h^{d-1}} \|p^*\|_{1, \infty, K} \label{key-ineq-1}.
\end{eqnarray}
\end{lemma}
The proof of Lemma \ref{p-w-lem} can be obtained by using Lemma \ref{vh-eq} and the proof of  Theorem 3.1 in \cite{ch03}.

We have the following convergence theorem on the periodic case.
\begin{theorem}
\label{converg-per-thm}
Let $\big((\se, \ue), \pe\big)$ be the solution of (\ref{exp-fem-cont}) and  $(\theta_h^l, u_h^l), p_h^l, \lambda_h^l )$ be the solution of (\ref{exp-fem-hybrid-local}).
Then
\begin{eqnarray}
\begin{split}
&\|\big((\se, \ue), \pe\big)-\big((\theta_h^l, u_h^l), p_h^l\big)\|_{X\times Q}+\|\mathcal{P}_{\partial} \pe -\lambda_h^l\|_{-{1\over 2}, h}\\
&\leq Ch\|f\|_{1,\Omega}+C(\epsilon+h+\sqrt{\epsilon h})\|p^*\|_{2,\Omega}+C\sqrt{\epsilon \over h} \|p^*\|_{1, \infty, \Omega}. \label{error-fgp}
\end{split}
\end{eqnarray}
\end{theorem}
\begin{proof}
By Theorem \ref{equi-thm-local}, we can use problem
(\ref{exp-fem-num-local}) to perform an error analysis
for $\big((\theta_h^l, u_h^l), p_h^l \big)$.
Due to Theorem~\ref{inf-sup-thm}, it suffices to choose
$\big( (\xi_h, v_h), q_h)\big)\in X_{h}^l\times Q_h$
such that the right-hand side of (\ref{cea-est1}) is small.

Set $q_h=\mathcal{P}_h \pe$, the $L^2$ projection of $\pe$ onto $Q_h$.  Then
Poincar\'{e}-Friedrichs  inequality implies
\begin{equation}
\label{p-est}
\|\pe-q_h\|_{0, \Omega}\leq C h |\pe|_{1, \Omega}.
\end{equation}

We choose $\xi_h|_K=\nabla \we$, and the $\xi_h \in X_{1,h}^l$ by Lemma \ref{gradw-lem}. Due to Lemma \ref{p-w-lem},
\begin{eqnarray}
\label{gradp-est}
\begin{split}
\|\se-\xi_h\|_{0, \Omega}&= \big(\sum_K |\pe-\we|_{1, K}^2\big)^{1\over 2}
\leq C(\epsilon+h+\sqrt{\epsilon h})\|p^*\|_{2, \Omega} + C\sqrt{\epsilon \over h} \|p^*\|_{1, \infty, \Omega}.
\end{split}
\end{eqnarray}

Let $v_h:= \mathcal{I}_h u^*$. Note that $div(\mathcal{I}_h u^*)|_K=\langle f\rangle_K$, the integral average of $f$ over $K$.   Then
\begin{equation}
\label{divv-est}
\|div(\ue-v_h)\|_{0, \Omega}=\big(\sum_K\|f-\langle f \rangle_K\|_{0,K}^2)^{1\over 2} \leq Ch|f|_{1, \Omega}.
\end{equation}

Moreover, by Lemma \ref{vh-eq} and (\ref{key-ineq-1})
\begin{eqnarray}
\begin{split}
\|\ue-v_h\|_{0,K}&=\|\ke \nabla \pe - \ke\nabla \we|_{0, K}\\
&\leq  C (\epsilon+h+\sqrt{\epsilon h})\|p^*\|_{2, K} +C\sqrt{\epsilon h^{d-1}} \|p^*\|_{1, \infty, K}.
\end{split}
\end{eqnarray}
Consequently, it follows immediately that
\begin{eqnarray}
\label{v-L2-est}
\|\ue-v_h\|_{0,\Omega}\leq C(\epsilon+h+\sqrt{\epsilon h})\|p^*\|_{2, \Omega} + C\sqrt{\epsilon \over h} \|p^*\|_{1, \infty, \Omega}.
\end{eqnarray}
Combining (\ref{p-est}), (\ref{gradp-est}), (\ref{divv-est}) and (\ref{v-L2-est}), it follows
\begin{equation}
\label{converg-pgv}
  \|\big((\se, \ue), \pe\big)-\big((\theta_h^l, u_h^l), p_h^l\big)\|_{X\times Q}\leq  Ch\|f\|_{1,\Omega}+C(\epsilon+h+\sqrt{\epsilon h})\|p^*\|_{2,\Omega}+C\sqrt{\epsilon \over h} \|p^*\|_{1, \infty, \Omega}.
\end{equation}

Next  we employ the technique used in \cite{ab85} to estimate $\|\mathcal{P}_{\partial} \pe -\lambda_h^l\|_{-{1\over 2}, h}$.
Let $\tilde{\tau}_h:=|e|(\mathcal{P}_{\partial} \pe -\lambda_h^l)\psi_{\chi_e}^K\in \tilde{X}_{2,h}^l(K)$. Then
\[
\tilde{\tau}_h\cdot n_e=\mathcal{P}_{\partial} \pe -\lambda_h^l \quad  \text{on $e$ and $0$ otherwise}.
\]
By (\ref{basis-est}), it follows
\begin{equation}
\label{mult-00}
\|\tilde{\tau}_h\|_{0,K} +h_K \|div \tilde{\tau}_h\|_{0, K} \leq C h_K^{1\over 2} \|\mathcal{P}_{\partial} \pe -\lambda_h^l\|_{0,e}.
\end{equation}
Define $\tau_h=\tilde{\tau}_h$ in K and $\tau_h=0$ in $\Omega\setminus$ K.
By the second equation in (\ref{exp-fem-hybrid-local}), we have
\begin{equation}
\label{mult-01}
(\theta_h^l, \tilde{\tau}_h)_K +(p_h^l, div\tilde{\tau}_h)_K=(\mathcal{P}_{\partial} \pe -\lambda_h^l, \lambda_h^l)_e.
\end{equation}
Since $\se=\nabla \pe$, Green's formula  gives
\begin{equation}
\label{mult-02}
(\se, \tilde{\tau}_h)_K +(\pe, div\tilde{\tau}_h)_K=(\mathcal{P}_{\partial} \pe -\lambda_h^l, \pe)_e.
\end{equation}
By using (\ref{mult-01}), (\ref{mult-02}) and (\ref{mult-00}),  we get
\begin{eqnarray}
\label{mult-key}
\begin{split}
&\|\mathcal{P}_{\partial} \pe -\lambda_h^l\|_{0,e}^2 =(\mathcal{P}_{\partial} \pe -\lambda_h^l, \mathcal{P}_{\partial} \pe -\lambda_h^l)_e=( \pe -\lambda_h^l, \mathcal{P}_{\partial} \pe -\lambda_h^l)_e\\
&=(\se-\theta_h^l, \tilde{\tau}_h)_K+(\pe-p_h^l, div\tilde{\tau}_h)_K\\
&\leq \|\se-\theta_h^l\|_{0,K} \|\tilde{\tau}_h\|_{0,K} +\|\pe-p_h^l\|_{0,K} \|div\tilde{\tau}_h\|_{0,K}\\
&\leq C\big( h^{-{1\over 2}}\|\pe-p_h^l\|_{0,K}+h^{1\over 2} \|\se-\theta_h^l\|_{0,K}\big) \|\mathcal{P}_{\partial} \pe -\lambda_h^l\|_{0,e},
\end{split}
\end{eqnarray}
which gives
\[
\|\mathcal{P}_{\partial} \pe -\lambda_h^l\|_{0,e}\leq C\big( h^{-{1\over 2}}\|\pe-p_h^l\|_{0,K}+h^{1\over 2} \|\se-\theta_h^l\|_{0,K}\big).
\]
Consequently,
\begin{eqnarray}
\label{mult-03}
\begin{split}
&\|\mathcal{P}_{\partial} \pe -\lambda_h^l\|_{-{1\over 2}, h}^2=\sum_{K\in \mathfrak{T}} \sum_{e\subset \pK} h_K \|\mathcal{P}_{\partial} \pe -\lambda_h^l\|_{0,e}^2\\
&\leq C\big (\sum_K h_K^{-1} h_K \|\pe-p_h^l\|_{0,K}^2 + \sum_{K}  h_K^2 \|\se-\theta_h^l\|_{0,K}^2\big)\\
&\leq C \|\pe-p_h^l\|_{0, \Omega}^2 +Ch^2 \|\se-\theta_h^l\|_{0, \Omega}^2.
\end{split}
\end{eqnarray}

Owing to Theorem 4.6 in \cite{ga99}, we have
\begin{equation}
\label{cea-est-p}
\|\pe-p_h^l\|_{0, \Omega} \leq C\big\{ \|(\se, \ue)-(\theta_h^l, u_h^l)\|_X + \inf_{q_h\in Q_h} \|\pe-q_h\|_{0, \Omega} \big\}.
\end{equation}
Combining (\ref{mult-03}), (\ref{cea-est-p}) and (\ref{converg-pgv}) gives
\begin{eqnarray}
\|\mathcal{P}_{\partial} \pe -\lambda_h^l\|_{-{1\over 2}, h}\leq Ch\|f\|_{1,\Omega}+C(\epsilon+h+\sqrt{\epsilon h})\|p^*\|_{2,\Omega}+C\sqrt{\epsilon \over h} \|p^*\|_{1, \infty, \Omega}.
\end{eqnarray}
This completes the proof.
\end{proof}

\begin{corollary}
Let $(\se, \ue)$ be the solution of (\ref{exp-fem-cont}) and  $(\theta_h^l, u_h^l)$ be the solution of (\ref{exp-fem-num-local}).
Then
\begin{eqnarray*}
\|(\se, \ue)-(\theta_h^l, u_h^l)\|_{X}\leq Ch|f|_{1,\Omega}+C(\epsilon+h+\sqrt{\epsilon h})\|p^*\|_{2,\Omega}+C\sqrt{\epsilon \over h} \|p^*\|_{1, \infty, \Omega}
\end{eqnarray*}
\end{corollary}
\begin{proof}
Due to Theorem 4.5 in \cite{ga99}, it follows that
\[
\|(\se, \ue)-(\theta_h^l, u_h^l)\|_{X}\leq C\big\{ \inf_{(\xi_h, v_h)\in X_h^l} \|(\se, \ue)-(\xi_h, v_h)\|_{X}\big\}.
\]
Then using the proof of (\ref{error-fgp}) completes the proof immediately.
\end{proof}


\subsubsection{Analysis of expanded mixed MsFEM using oversampling techniques}

By Theorem \ref{converg-per-thm}, the local expanded mixed MsFEM without using oversampling techniques causes
a resonance error $O(\sqrt{\epsilon\over h})$.
By oversampling technique, the resonance error $O(\sqrt{\epsilon\over h})$ in Theorem (\ref{converg-per-thm})
will reduce to $O({\epsilon\over h})$.  In this subsection we  sketch  the analysis for
the oversampling expanded mixed MsFEM.

Let $S\supset K$ be the local domain described in equation (\ref{os-basis}).  We assume that  $dist(\pK, \partial S)\approx h_K$ for the analysis.
Let constant $c_{jl}$ be defined in (\ref{const-c-os}).
The constants $c_{jl}$ allow us to introduce an extension operator ${E}^S$ to $S$ from $K$, e.g.,
$E^S: V_h(K) \longrightarrow V_h(S)$ is defined by
\[
E^S(\sum_{j} \beta_j \chi_j^K):=\sum_{j}\sum_l \beta_j c_{jl}\chi_l^S.
\]
Similarly, $E^S: \tilde{X}_{2,h}^{os}(K) \longrightarrow X_{2,h}^l(S)$ is defined by
\[
E^S(\sum_{j} \beta_j \bar{\psi}_{\chi_j}^K):=\sum_{j}\sum_l \beta_j c_{jl}\psi_{\chi_l}^S.
\]
Since $dist(\pK, \partial S)\approx h_K$,  the extension operator $E^S$ has the property
\[
\|E^S(v^K)\|_{0,S}\approx \|v^K\|_{0,K},  \quad \text{for}  \quad \forall v^K\in V_h(K) \quad \text{or} \quad \tilde{X}_{2,h}^{os}(K).
\]

In order to analyze convergence, we define  $\wes(x)$  be the solution (up to a constant)  of the following equation
\begin{eqnarray}
\label{w-eq-os}
\begin{split}
\begin{cases}
-div (\ke(x)\nabla \wes) &=div (E^S(\mathcal{I}_h^* u^*|_K)) \ \  \text{in} \ \ S\\
-\ke(x)\nabla \wes \cdot n_{\pS} &=E^S(\mathcal{I}_h^* u^*|_K)\cdot n_{\pS} \ \ \text{on} \ \ \pS.
\end{cases}
\end{split}
\end{eqnarray}
Here we note that $\mathcal{I}_h^* u^*|_K$ is the local interpolation on $V_h(K)$ (i.e.,  $RT_0(K)$).
We define the multiscale interpolation operator  $\mathcal{I}_h^{os}: [H^1(\Omega)]^d\longrightarrow \tilde{X}_{2,h}^{os}$ for velocity by
\[
\mathcal{I}_h^{os}(v):=\sum_{e\in \mathcal{F}_h} (\int_e v\cdot n_{e} ds) \bar{\psi}_{\chi_e}.
\]
Then by using a procedure similar to Lemma \ref{vh-eq}, Lemma \ref{gradw-lem} and  straightforward calculations,  we obtain  the following lemma.
\begin{lemma}
\label{vh-eq-os}
Let $\wes$  be defined in (\ref{w-eq-os}). Then
\begin{eqnarray}
\begin{split}
-\ke\nabla \wes =E^S(\mathcal{I}_h^{os} u^*|_K)  \ \ \text{on}  \ \ S,  \quad -\ke\nabla \wes|_K =\mathcal{I}_h^{os} u^*|_K  \ \ \text{on}  \ \ K.
\end{split}
\end{eqnarray}
Moreover, $\nabla \wes|_K =\sum_{e\subset \pK} (\int_e u^*\cdot n_e ds)\bar{\eta}_{\chi_e}^K\in X_{1,h}^{os}(K)$,
where $\bar{\eta}_{\chi_e}^K$ is defined in (\ref{def-bar-psi}).
\end{lemma}

We have the convergence result for the oversampling expanded mixed MsFEM.
\begin{theorem}
\label{converg-os-thm}
Let $((\se, \ue), \pe)$ be the solution of (\ref{exp-fem-cont}) and  $(\theta_h^{os}, u_h^{os}, p_h^{os},\lambda_h^{os} )$ be the solution of (\ref{exp-fem-hybrid-os}).
Then
\begin{eqnarray}
\label{error-fgp-os}
\begin{split}
&\sum_K \|\big((\se, \ue), \pe\big)-\big((\theta_h^{os}, u_h^{os}), p_h^{os}\big)\|_{X(K)\times Q(K)}+\|\mathcal{P}_{\partial} \pe -\lambda_h^{os}\|_{-{1\over 2}, h}\\
&\leq C(h+\epsilon) (\|f\|_{1,\Omega}+\|p^*\|_{2,\Omega}) +C({\epsilon \over h}+\sqrt{\epsilon})(\|p^*\|_{1,\infty, \Omega} +\|f\|_{0, \Omega}).
\end{split}
\end{eqnarray}
\end{theorem}

The proof of Theorem \ref{converg-os-thm} is presented  in Appendix \ref{app2}.

\subsection{Error analysis in $G$ convergence}

In subsection \ref{peroidc-case}, we have investigated the case when $\ke$ in  (\ref{target-elliptic}) is $\epsilon$ periodic.  However,
the expanded  mixed MsFEMs   can also  be applied to non-periodic separable scales.  In this section
we will discuss the convergence for the  case of separable scales described in   $G-$convergence (ref. \cite{jko94}).
$G-$convergence is more general than the periodic homogenization described in Subsection \ref{peroidc-case}.

A sequence of matrices $\ke$  is $G-$convergent to $k^*$
if for any open set $\omega\subset \Omega$ and any right-hand side  $f\in H^{-1}(\omega)$ in (\ref{target-elliptic}), if the sequence of the solutions
$\pe$ in (\ref{target-elliptic}) satisfies
\[
\pe \rightharpoonup p^* \ \  \text{weakly in} \ \ H^1(\omega) \ \ \text{as} \ \ \epsilon\rightarrow 0,
\]
where $p^*$ is the solution of the equation (\ref{homog-elliptic}), in which $k^*$ is the homogenized matrix in the sense of $G-$convergence.
The $G-$convergence implies that
\[
\ke \nabla \pe \rightharpoonup k^* \nabla p^* \ \ \text{weakly in} \ \ L^2(\omega) \ \ \text{as} \ \ \epsilon\rightarrow 0.
\]
There is no explicit formula for the matrix  $k^*$, which is defined as a limit in the distributional sense, i.e.,
\[
\ke\nabla \mathcal{N}^i_{\epsilon} \rightharpoonup k^* e_i \ \ \text{in} \ \ \mathcal{D}'(\omega; \mathbb{R}^d),
\]
where the auxiliary functions $\mathcal{N}^i_{\epsilon}$ ($i=1,\cdots, d$) satisfy
\[
\mathcal{N}^i_{\epsilon}\rightharpoonup x_i \ \  \text{weakly in} \ \ H^1(\omega) \ \ \text{as} \ \ \epsilon\rightarrow 0.
\]
The  auxiliary functions are not explicit.
We define the corrector matrix $\nabla \mathcal{N}_{\epsilon}=(\frac{\partial \mathcal{N}^i_{\epsilon}}{\partial x_j})_{i,j=1,\cdots, d}$.

The following lemma is about corrector in $G-$convergence.
\begin{lemma} \cite{mt97}
\label{corrector-est-G}
Let $\ke$ be a sequence $G-$converging to $k^*$ as $\epsilon\rightarrow 0$. Then
\[
\nabla \pe=\nabla \mathcal{N}_{\epsilon} \cdot \nabla p^* +R_{\epsilon}^{\omega},
\]
where  $R_{\epsilon}^{\omega}\rightarrow 0$ strongly in $L^1(\omega)$
as $\epsilon\rightarrow 0$.
Moreover, if $\nabla \mathcal{N}_{\epsilon}$ is bounded in $L^r(\omega)$ for some $r$ such that $2\leq  r \leq \infty$,
and $\nabla p^*\in L^s(\omega)$ for some $s$ such that $2\leq s< \infty$,  then
$R_{\epsilon}^{\omega}\rightarrow 0$ strongly in $L^t(\omega)$, as $\epsilon\rightarrow 0$, where
$t=\min\{2,\frac{rs}{r+s}\}$.
\end{lemma}
We have the following convergence theorem for the case of  $G-$convergence.
\begin{theorem}
\label{converg-G-thm}
Let $\big((\se, \ue), \pe\big)$ be the solution of (\ref{exp-fem-cont}) and  $(\theta_h^l, u_h^l, p_h^l, \lambda_h^l )$ be the solution of (\ref{exp-fem-hybrid-local}).
If $\nabla \mathcal{N}_{\epsilon}\in L^{\infty} (K)$ for all $K$, then
\begin{eqnarray}
\label{error-fgp-G}
\begin{split}
 &\lim_{\epsilon \rightarrow 0}\big \{ \|\big((\se, \ue), \pe\big)-\big((\theta_h^l, u_h^l), p_h^l\big)\|_{X\times Q} + \|\mathcal{P}_{\partial} \pe -\lambda_h^l\|_{-{1\over 2}, h} \big\}\\
 &\leq Ch(|u^*|_{0,\Omega} +\|f\|_{1, \Omega}).
\end{split}
\end{eqnarray}
\end{theorem}
\begin{proof}
The proof of (\ref{error-fgp-G}) is similar to the proof of (\ref{error-fgp}).
We set $q_h=\mathcal{P}_h \pe$, $\xi_h|_K=\nabla \we$ and $v_h=\mathcal{I}_h u^*$ and then use Theorem \ref{inf-sup-thm}.
By the proof of (\ref{error-fgp}), it immediately follows
\begin{eqnarray}
\label{ineq-1-G}
\|div(\ue-v_h)\|_{0, \Omega} +\|\pe-q_h\|_{0, \Omega}\leq Ch(|\pe|_{1, \Omega}+|f|_{1,\Omega})\leq Ch \|f\|_{1, \Omega}.
\end{eqnarray}
By  Lemma \ref{vh-eq} and Lemma \ref{corrector-est-G}
\begin{eqnarray}
\begin{split}
&\|\nabla \pe-\nabla \we\|_{0,K}\\
&\leq \|\nabla \pe-\nabla \mathcal{N}_{\epsilon} \nabla p^*\|_{0,K} + \| \nabla \mathcal{N}_{\epsilon}(\nabla p^*-\nabla \wh)\|_{0,K}+\|\nabla \mathcal{N}_{\epsilon}\nabla \wh-\nabla \we\|_{0.K}\\
&\leq \|\nabla \pe-\nabla \mathcal{N}_{\epsilon} \nabla p^*\|_{0,K}+  \|\nabla \mathcal{N}_{\epsilon}\big( (k^*)^{-1}(u^*-\mathcal{I}_h^* u^*)\big )\|_{0, K} +\|R_{\epsilon}^K\|_{0, K}\\
&\leq \|\nabla \pe-\nabla \mathcal{N}_{\epsilon} \nabla p^*\|_{0,K} +C h |u^*|_{0,K} +\|R_{\epsilon}^K\|_{0, K},
\end{split}
\end{eqnarray}
which gives
\begin{eqnarray}
\label{ineq-2-G}
\|\nabla \pe-\nabla \xi_h\|_{0,\Omega}\leq C\|R_{\epsilon}^{\Omega}\|_{0, \Omega} +C h |u^*|_{0,\Omega}+C\sum_K \|R_{\epsilon}^K\|_{0, K}.
\end{eqnarray}
Similarly we have
\begin{eqnarray}
\label{ineq-3-G}
\| \ue-v_h\|_{0,\Omega}\leq C\|R_{\epsilon}^{\Omega}\|_{0, \Omega} +C h |u^*|_{0,\Omega}+C\sum_K \|R_{\epsilon}^K\|_{0, K}.
\end{eqnarray}
By combining (\ref{ineq-1-G}), (\ref{ineq-2-G}), (\ref{ineq-3-G}) and Lemma \ref{corrector-est-G} and letting $\epsilon \rightarrow 0$, we have
\begin{eqnarray}
\label{ineq-4-G}
 \lim_{\epsilon \rightarrow 0} \|\big((\se, \ue), \pe\big)-\big((\theta_h^l, u_h^l), p_h^l\big)\|_{X\times Q}
 \leq Ch(|u^*|_{0,\Omega} +\|f\|_{1, \Omega}).
\end{eqnarray}
By following the proof of Theorem \ref{converg-per-thm} and
using (\ref{ineq-4-G}), we can immediately obtain that
\begin{eqnarray}
\label{ineq-5-G}
\lim_{\epsilon \rightarrow 0}\|\mathcal{P}_{\partial} \pe -\lambda_h^l\|_{-{1\over 2}, h}
 \leq Ch(|u^*|_{0,\Omega} +\|f\|_{1, \Omega}).
 \end{eqnarray}
 The proof of (\ref{error-fgp-G}) is completed by combining (\ref{ineq-4-G}) and (\ref{ineq-5-G}).
\end{proof}

Because no explicit formula is available for homogenization matrix $k^*$,  we are not able to obtain an explicit convergence rate in $G-$convergence.
The approximation rate is only presented in terms of limit as $\epsilon \rightarrow 0$. We can also obtain the convergence of oversampling  expanded mixed MsFEM in terms
of $G-$convergence, and the convergence result is the same as in Theorem \ref{converg-G-thm}.


\subsection{Error analysis for global expanded mixed MsFEM}

In this subsection, we  present the convergence results for the expanded mixed MsFEM using global information.
 We consider the continuum scales for the global expanded mixed MsFEM.

For analysis, we make the following assumption for global information.
\begin{assumption}
\label{assum-G}
There exist global fields $\{u_1, \cdots, u_N\}$ and continuous functions $A_i(x)$ ($i=1,\cdots N$)  such that
\[
u=\sum_{i=1}^N A_i(x) u_i, \quad  A_i(x)\in C^{\alpha}.
\]
\end{assumption}

The existence of these global fields has been discussed in the literature recently~(see e.g., \cite{aej08, bo09, jem10, oz07}).
In particular, Owhadi and Zhang \cite{oz07} develop global fields $(u_1, \dots, u_d)$ that are solutions of the
following equations
\begin{eqnarray}
\label{harmonic-eq}
\begin{cases}
\begin{split}
u_i+k\theta_i &=0 \ \   \text{in} \ \ \Omega\\
\theta_i-\nabla p_i &=0  \ \   \text{in} \ \ \Omega\\
div (u_i) &= 0 \ \ \text{in} \ \ \Omega \\
p_i&= x_i \ \ \text{on} \ \ \partial \Omega \, ,
\end{split}
\end{cases}
\end{eqnarray}
where $x=(x_1,\dots, x_d)$.  Using these global fields, and Theorem 1.3
in \cite{oz07}, we can prove the following proposition.
\begin{proposition}
\label{global-fields-prop}
Let $(\theta, u)$ be the solution  of (\ref{1st-order-target}) and
$(\theta_i, u_i, p_i)$ solve equation (\ref{harmonic-eq})
for $i=1, \dots, d$.
Let $\mathbb{P}=(p_1, \dots, p_d)$. If $(\nabla \mathbb{P})^T k \nabla\mathbb{P}$ satisfies the anisotropic condition
described in Theorem 1.3 in  \cite{oz07}, then
\[
u=\sum_{i=1}^d A_i(x)u_i, \quad \theta=\sum_{i=1}^d A_i(x) \theta_i,
\]
where $A_i(x)$ are continuous functions.  Moreover, for any $s>d$, $i=1, \dots, d$,
\[
\|A_i\|_{C^{1-{d\over s}}(\Omega)}^2\leq C \|f\|_{0,s,\Omega}^2,
\]
where the constant $C$ depends on the anisotropic condition and $s$.
\end{proposition}

\begin{remark}
  To simplify computation of global fields and the coupled
  coarse-scale system by expanded mixed MsFEMs, we may use a single
  global field ($N=1$) to construct the basis functions. This simplification
  has been shown to be effective for relevant two-phase flow simulations
  \cite{aej08, jem10}.
\end{remark}

Let $\mathcal{I}_h^g u$ be the interpolation of $u$, with its restriction on $K$ defined by
\[
\mathcal{I}_h^g u|_K=\sum_{i=1}^N \sum_{e\subset \pK}(\int_{e}A_i u_i \cdot n_e ds) \psi_{i, \chi_e}^K,
\]
where the velocity basis function $\psi_{i, \chi_e}^K$ is defined in  (\ref{global-basis}).
Then it follows that
\[
(div(u-\mathcal{I}_h^g u), q_h)=0  \quad \text{for} \quad \forall q_h\in Q_h.
\]
By Theorem 3.4 in \cite{aej08}, there exists $0<\mu\leq 1$, which depends on $\alpha$ and $d$, such that
\begin{equation}
\label{interp-global}
\|u-\mathcal{I}_h^g u\|_{0, \Omega}\leq Ch^{\mu}(\sum_{i} \|A_i\|_{C^{\alpha}(\Omega)}).
\end{equation}

Let $w(x)$  be the solution (up to a constant)  of the following equation
\begin{eqnarray}
\label{w-eq-global}
\begin{split}
\begin{cases}
-div (k\nabla w) &=div (\mathcal{I}_h^g u|_K) \ \  \text{in} \ \ K\\
-k\nabla w \cdot n_{\pK} &=(\mathcal{I}_h^g u|_K)\cdot n_{\pK} \ \ \text{on} \ \ \partial K.
\end{cases}
\end{split}
\end{eqnarray}
Then we have the following lemma.
\begin{lemma}
\label{w-lem-g}
Let $w$ be solution of (\ref{w-eq-global}). Then
\[
-k(x)\nabla w=\mathcal{I}_h^g u|_K \quad \text{on $K$}.
\]
Moreover, $\nabla w\in X_{1,h}^g(K)$.
\end{lemma}

The proof of Lemma \ref{w-lem-g} is given  in Appendix \ref{app3}.
By using Lemma \ref{w-lem-g}, the interpolation estimate (\ref{interp-global}) and  the techniques in
the proof of Theorem \ref{converg-per-thm}, we can derive the convergence result for the global expanded mixed MsFEM.
\begin{theorem}
\label{converg-global-thm}
Let $\big((\se, \ue), \pe\big)$ be the solution of (\ref{exp-fem-cont}) and  $(\theta_h^g, u_h^g, p_h^g, \lambda_h^g)$ be the solution of (\ref{exp-fem-hybrid-global}).
Under Assumption \ref{assum-G},  there exists $\mu$ ($0<\mu\leq 1$) such that
\begin{eqnarray}
\|\big((\theta, u), p\big)-\big((\theta_h^g, u_h^g), p_h^g\big)\|_{X\times Q} + \|\mathcal{P}_{\partial} p -\lambda_h^g\|_{-{1\over 2}, h}  &\leq & Ch^{\mu} \label{error-fgp-global},
\end{eqnarray}
where $C=C(\|f\|_{1, \Omega})$.
\end{theorem}

Theorem \ref{converg-global-thm} shows  that the convergence of global expanded  mixed MsFEM is independent of
small scales, and  the resonance error is removed using the global information.


\section{Numerical results}

In this section, we present numerical results that highlight the
advantages of the proposed expanded mixed MsFEMs.  In particular, we
show the improvement in the accuracy of the gradient unknown $\theta$
and the velocity unknown $u$ obtained with this new family of methods.
In addition, we confirm that if the media has strong non-local
features (e.g., channels and fractures), then the use of global
information in the construction of the basis functions significantly
improves the accuracy of the multiscale solution

In all of the numerical experiments reported below, we use $RT_0$ in
three dimensions, both as the fine-scale velocity space and as
fine-scale shape functions for the space approximating $\nabla p$ in
$L^2(\Omega)^3$ (not conforming in $H(div, \Omega)$) . The domain
$\Omega$ is discretized by a uniform hexahedral fine mesh of size $h$
and a uniform hexahedral coarse mesh of size $H>h$.  The fine mesh
is nested in the coarse mesh. The reference solutions for pressure,
velocity, and $\nabla p$ are computed by solving the fine-scale
expanded MFEM.

Our primary focus in this section is on the behavior of the error of
the MsFEM solutions for the velocity and pressure gradient.  These
errors are significantly affected by features of the multiscale models
and by the different boundary conditions that are used in the
construction of the multiscale basis functions.  The velocity is
particularly important and required for multi-phase simulations.  All
errors reported in this section are measured in $L^2$.  We are also
interested in the practical application of simulating multi-phase
flow. Therefore, in addition to velocity error data, we present
solutions obtained by the IMPES (implicit pressure, explicit
saturation) method for the coupled saturation and pressure equations
of two-phase flow (incompressible and immiscible).

\subsection{Low permeability channel}
\label{sec:ex-channel}

In the first experiment, the domain $\Omega$ is the unit cube.  We
take $h=1/24$ and $H=1/8$, so that each coarse element contains
$3\times 3\times 3$ fine elements.  The source function is as follows:
\begin{equation}\label{source}
f(x,y,z) = \left\{ \begin{array}{rl} 1 & \text{if }(x,y,z)\text{ is in } (0,H)^3, \\
-1 &  \text{if }(x,y,z)\text{ is in } (1-H,1)^3, \\
0 & \text{otherwise.}
\end{array} \right.
\end{equation}
This source function represents injection and production wells in
opposite corners of the domain, occupying coarse elements. It is a
practical example that demonstrates corner to corner flow across the
domain, allowing us to test different numerical methods in simulating
flow for various permeability fields.

In the case that the permeability is positive and bounded well above
zero everywhere, the expanded
mixed MsFEM gives comparable (almost identical) results to the standard
mixed MsFEM. However, if the permeability is not positive
everywhere, the standard mixed MsFEM is not applicable. Moreover, if the
permeability is not bounded well above zero everywhere, the finite
precision may significantly degrade the its performance.
Table \ref{lowpermchannel} demonstrates this fact, where a test is reported
using a channel of low permeability. This model problem does not have well separated
scales.  The channel has one fine cell's width in the $y-$ and $z-$directions,
and it extends across the entire domain in the $x-$direction. More
specifically, the channel occupies $(0,1)\times (\frac{4}{24},
\frac{5}{24})\times (\frac{7}{24}, \frac{8}{24})$.  Inside the
channel, $k=10^{-4}$, and elsewhere $k=1$.

The label ``MsFEM'' in Table \ref{lowpermchannel} represents the
standard mixed MsFEM, with variables for pressure and velocity
only. This method makes no direct approximation of the pressure
gradient, but we approximate $\nabla p$ with velocity divided by
permeability, when the permeability is positive everywhere. ``Expanded
MsFEM'' is the expanded mixed MsFEM, which approximates the pressure
gradient $\nabla p$ locally by non-conforming Raviart-Thomas
elements. ``GMsFEM'' means that global information, specifically the
fine-scale velocity obtained from the expanded mixed FEM, was used in
constructing the coarse multiscale basis functions. We observe that
using global information reduces the error by more than a factor of 20
in both the velocity and the pressure gradient, while the difference
between standard MsFEM and expanded MsFEM is negligible.
In the remainder
of this section, we do not report results for standard MsFEM.

\begin{table}
\caption{Low permeability channel}
\label{lowpermchannel}
\begin{center}
\begin{tabular}{|c|c|c|c|c|}
\hline
& MsFEM & GMsFEM & Expanded MsFEM & Expanded GMsFEM \\
\hline \hline
Velocity error & $1.173\times 10^{-3}$ & $5.005\times 10^{-5}$ & $1.173\times 10^{-3}$ & $5.005\times 10^{-5}$  \\ \hline
Velocity norm & $5.649\times 10^{-3}$ & $5.649\times 10^{-3}$ & $5.649\times 10^{-3}$ & $5.649 \times 10^{-3}$\\ \hline
$\nabla p$ error & $1.739\times 10^{-3}$ & $5.284\times 10^{-5}$ & $1.858\times 10^{-3}$  & $5.7812\times 10^{-5}$  \\ \hline
$\nabla p$ norm & $5.670\times 10^{-3}$ & $5.670\times 10^{-3}$ & $5.670\times 10^{-3}$ & $5.670\times 10^{-3}$ \\ \hline
\end{tabular}
\end{center}
\end{table}

\begin{table}
\caption{Channel with vanishing permeability}
\label{channelvanish}
\begin{center}
\begin{tabular}{|c|c|c|}
\hline
& Expanded MsFEM & Expanded GMsFEM \\
\hline \hline
Velocity error & $1.140\times 10^{-3}$ & $1.122\times 10^{-4}$ \\ \hline
Velocity norm & $5.642\times 10^{-3}$ & $5.642\times 10^{-3}$ \\ \hline
$\nabla p$ error & $1.461\times 10^{-3}$ & $9.338\times 10^{-4}$ \\ \hline
$\nabla p$ norm & $5.782\times 10^{-3}$ & $5.782\times 10^{-3}$ \\ \hline
\end{tabular}
\end{center}
\end{table}

\begin{figure}
  \begin{center}
    \includegraphics[height=50mm]{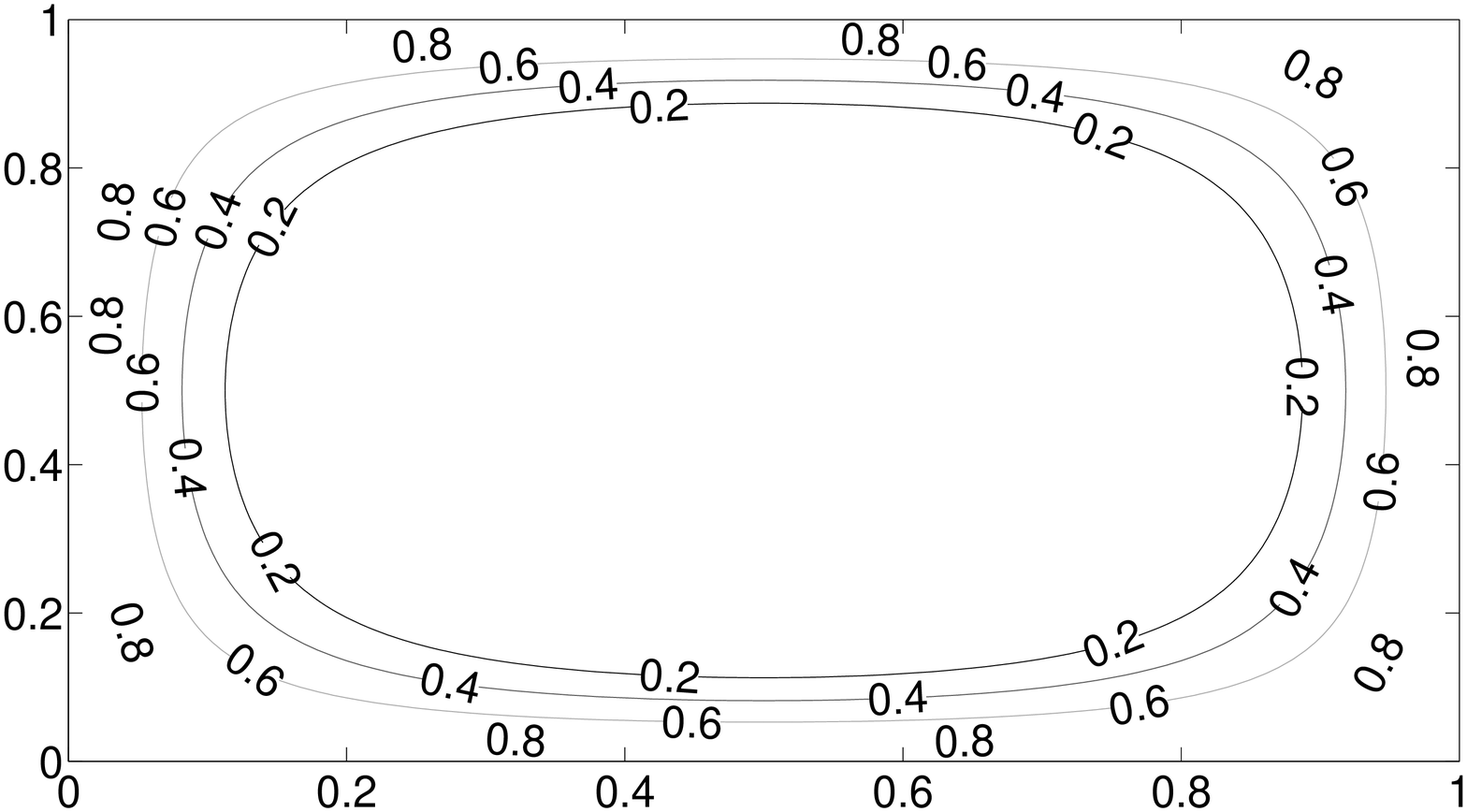}
  \end{center}
  \caption{Partially vanishing permeability within a reference fine cell}
  \label{vanishingpermcontour}
\end{figure}

Next, we test the more interesting case of a permeability that
vanishes in the channel. Standard mixed MsFEMs cannot be used in this
case. In each cell within the channel, we define the
permeability to be the vanishing, non-negative function
\begin{equation}
  \label{vanishingperm}
  k=\max (0,1-32x(1-x)y(1-y)),\qquad \text{for }(x,y)\in (0,1)^2,
\end{equation}
defined on the reference unit square (see
Figure~\ref{vanishingpermcontour} for the contour plot).  Outside the
channel, $k=1$.  This permeability has no scale separation, and it
ensures a positive definite mass matrix (weighted by $k$) for the
discontinuous space approximating $\nabla p$ in $L^2(\Omega)^3$. Note
that expanded mixed MsFEMs allow the permeability to vanish locally
within fine cells, unlike standard mixed MsFEMs which involve
$k^{-1}$, but the mass matrix weighted by $k$ must be invertible. The
results for the expanded methods are shown in Table
\ref{channelvanish}. Global information reduces the relative velocity error
from approximately $20\%$ to $2\%$ and the relative pressure gradient error
from approximately $25\%$ to $16\%$.


\subsection{Oscillatory permeability}
\label{sec:ex-oscillatory}

The next experiment demonstrates the resonance error that appears
in the velocity and the pressure gradient that are obtained using the
local expanded MsFEM.  The permeability is taken to be the oscillatory function
\begin{equation}
  k = (\sin(20\pi x)+1.5)(\sin(20\pi y)+1.5)
  \label{eq:oscillatory-perm}
\end{equation}
on the domain $\Omega=(0,1)\times (0,1)\times (0,\frac{1}{8})$, and
the mesh has $100\times 100\times 8$ fine cells.  The source function
is the two-dimensional analog of \eqref{source} (constant in the
$z$-direction), namely
\begin{equation}\label{source2d}
f(x,y,z) = \left\{ \begin{array}{rl} 1 & \text{if }(x,y)\text{ is in } (0,H)^2, \\
-1 &  \text{if }(x,y)\text{ is in } (1-H,1)^2, \\
0 & \text{otherwise.}
\end{array} \right.
\end{equation}
Thus the problem is essentially two-dimensional, although we solve in
a three-dimensional domain.  The solution will be approximately
constant in the $z$-direction, so we simply plot an $x$-$y$ slice.
We tested both local and global expanded mixed MsFEMs with
coarse meshes of size $20\times20 \times 1$ (each coarse cell having
$5\times5\times 8$ fine cells) and $10\times 10\times 1$ (each coarse
cell having $10\times10\times 8$ fine cells). In the $20\times
20\times 1$ case, the permeability has one period per coarse edge
length.  Table \ref{oscillatory100} lists the error for both methods
on the two choices of coarse grids.

\begin{table}
\caption{Oscillatory permeability without scale separation, $100\times 100\times 8$ fine mesh}
\label{oscillatory100}

\begin{center}
\begin{tabular}{|c|c|c|c|c|}
\hline
& \multicolumn{2}{|c|}{$20\times 20\times 1$ coarse mesh} & \multicolumn{2}{|c|}{$10\times 10\times 1$ coarse mesh} \\
\hline
& Exp. MsFEM & Exp. GMsFEM & Exp. MsFEM & Exp. GMsFEM \\
\hline \hline
Velocity error & $2.066\times 10^{-3}$ & $1.050\times 10^{-4}$ & $1.827\times 10^{-3}$ & $7.606\times 10^{-4}$ \\ \hline
Velocity norm & $5.143\times 10^{-3}$ & $5.143\times 10^{-3}$ & $5.143\times 10^{-3}$ & $5.143\times 10^{-3}$ \\ \hline
$\nabla p$ error & $1.370\times 10^{-3}$ & $1.452\times 10^{-4}$ & $1.552\times 10^{-3}$ & $5.360\times 10^{-4}$ \\ \hline
$\nabla p$ norm & $3.254\times 10^{-3}$ & $3.254\times 10^{-3}$ & $3.254\times 10^{-3}$ & $3.254\times 10^{-3}$ \\ \hline
\end{tabular}
\end{center}
\end{table}

The local expanded mixed MsFEM clearly suffers from the resonance errors in
the velocity and the pressure gradient.  In this example, refinement
of the coarse mesh does not change the error significantly because
the oscillations in $k$ are not resolved, even by the $20\times 20 \times 1$ coarse
mesh. In fact, the error of the velocity actually increases slightly when
the coarse grid is refined.
Figure \ref{lvx10and20} shows the $x$-component of the reference
velocity on the fine grid (top left), and the excellent agreement with
the $x$-component of the velocity obtained with the global expanded
mixed MsFEM on the $20\times 20 \times 1$ coarse grid (top right).  In
contrast, the impact of the resonance errors on the $x$-component of the
velocity obtained with the local expanded mixed MsFEM is evident
in Figure \ref{lvx10and20} for both the $10\times 10\times 1$
(bottom-left) and $20\times 20\times 1$ (bottom-right) coarse grids.
Note that these plots are in an $x$-$y$ slice, as the solutions
are constant in the $z$-direction.  On the other hand, the global
expanded mixed MsFEM performs well in this example, and the refinement of the coarse mesh
reduces the error by more than half.

\begin{figure}
  \centering
  \includegraphics[height=2in]{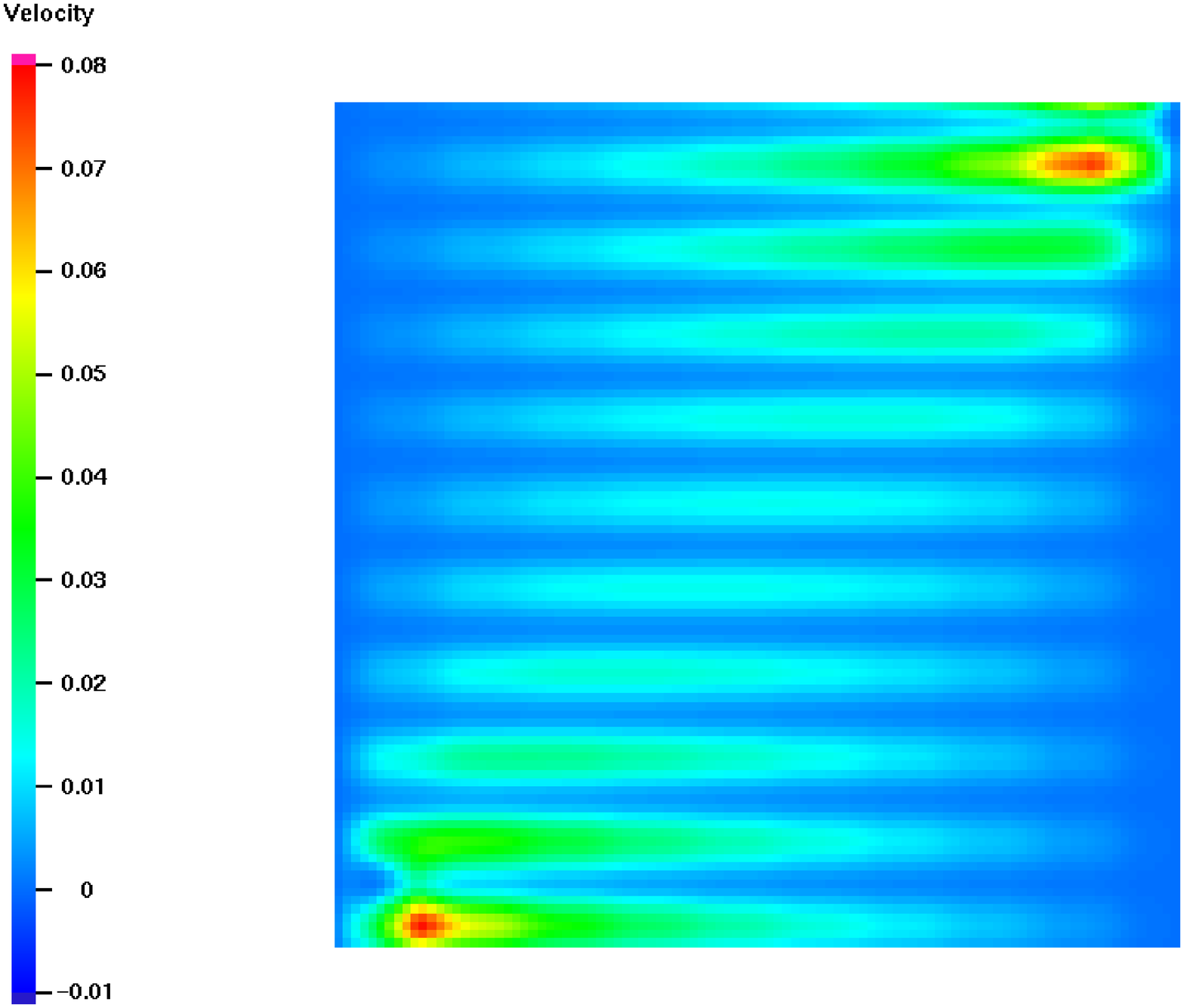}  \hspace*{12pt}
  \includegraphics[height=2in]{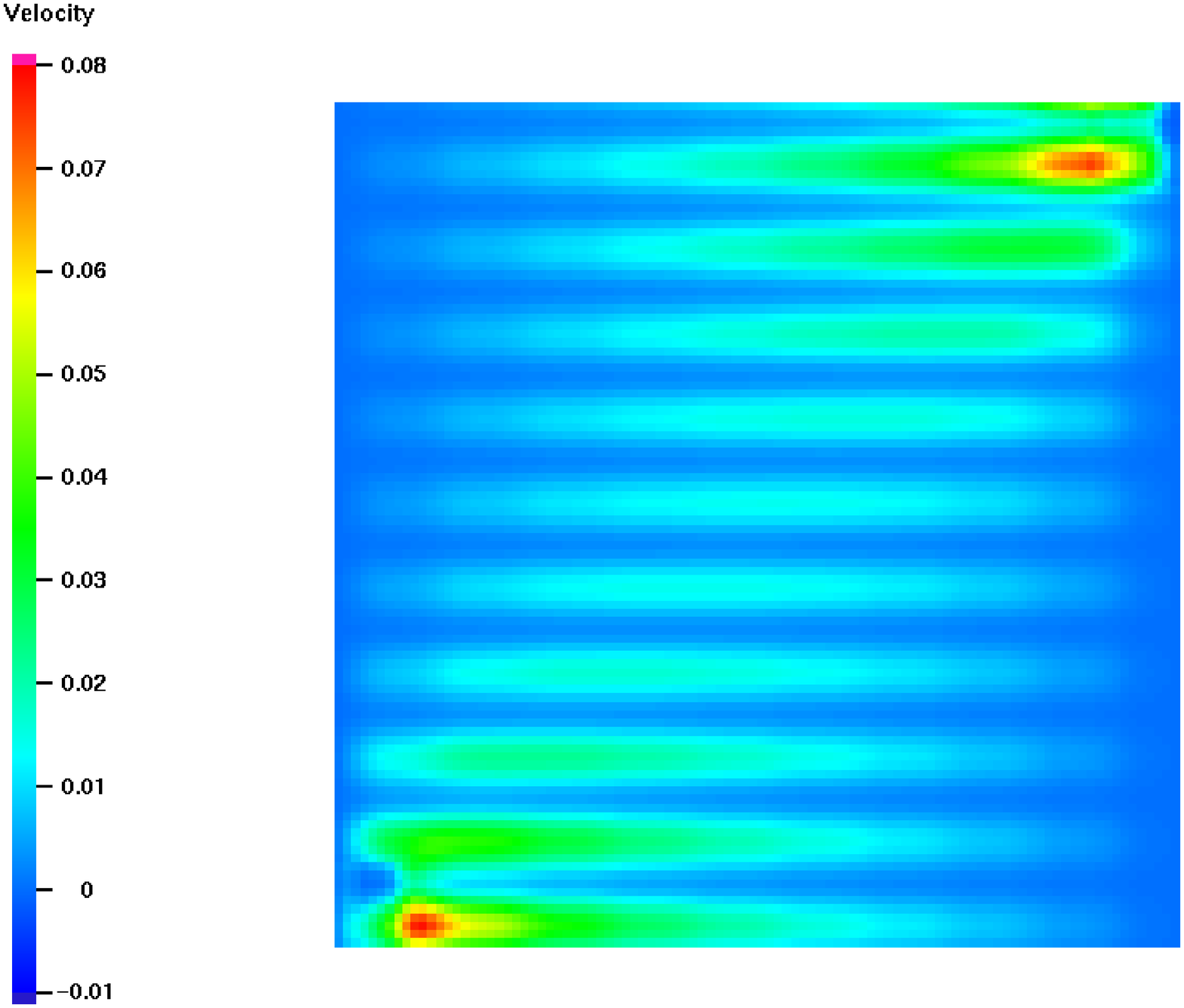} \\[12pt]
  \includegraphics[height=2in]{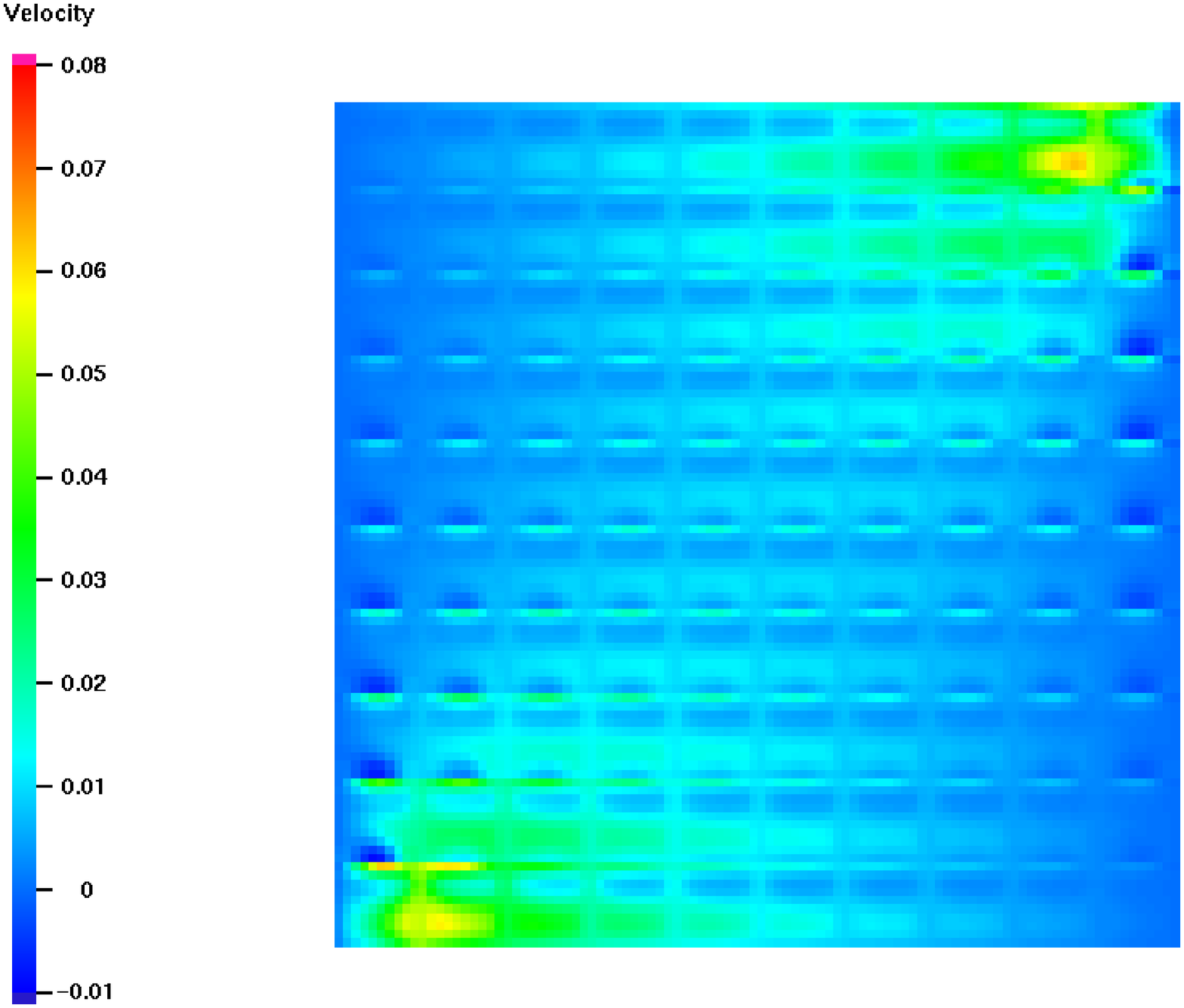}      \hspace*{12pt}
  \includegraphics[height=2in]{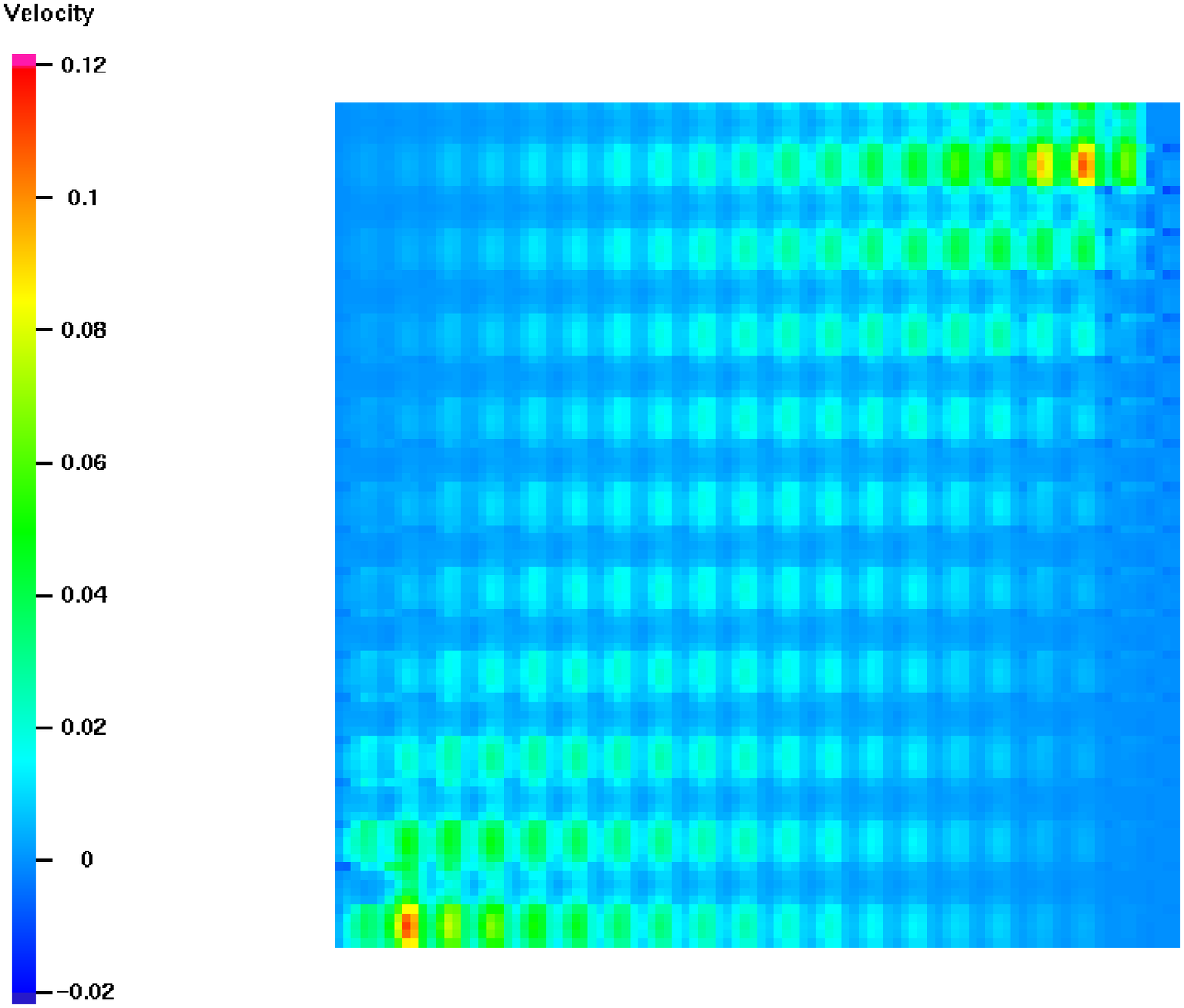}
  \caption{The $x$-component of the velocity for the oscillatory
    permeability field \protect\eqref{eq:oscillatory-perm} is plotted
    for the fine-scale reference solution (top-left), and for three
    expanded mixed MsFEMs.  The solution obtained with the global
    expanded mixed MsFEM on the $20\times 20\times 1$ coarse grid (top
    right) shows excellent agreement with the reference solution.  The
    solutions obtained with the local expanded mixed MsFEM on the $10
    \times 10 \times 1$ coarse grid (bottom left) and the
    $20\times20\times 1$ coarse grid (bottom right) show the impact of
    the resonance errors.}
  \label{lvx10and20}
\end{figure}

\subsection{IMPES with random shales}
\label{sec:ex-impes}

To compare the performance of the local and global expanded mixed MsFEMs for
a two-phase flow problem, we use an Implicit Pressure Explicit
Saturation (IMPES) formulation.  The boundary conditions that we use
to define the global expanded MsFEM basis functions are obtained from
the fine-scale expanded MFEM solution at the initial time.  We note
that in very complex highly varying flows it may be advantageous to
update these basis functions by updating the fine-scale expanded MFEM
solution that is used in their construction.  However, our current
study is focused on the effect of the two different velocity
approximations, and hence, we fix the multiscale bases throughout the
simulation.

\begin{figure}
  \centering
  \includegraphics[height=2.5in, width=5in]{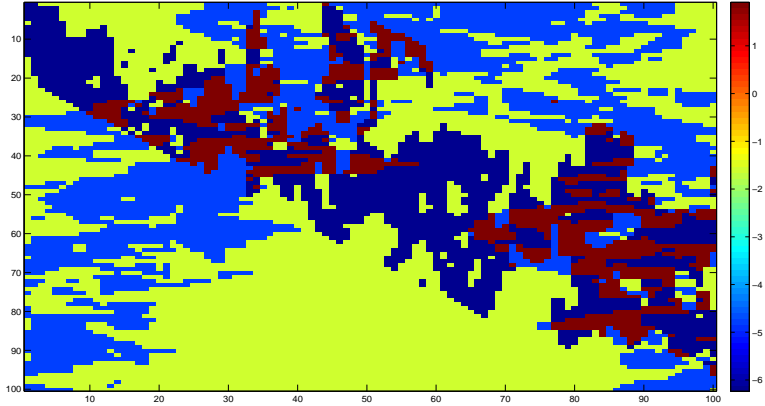}
  \caption{Logarithm of permeability field used in the two-phase flow two-spot problem.}
  \label{logperm}
\end{figure}

\begin{figure}
  \begin{center}
    \includegraphics[width=0.7\linewidth]{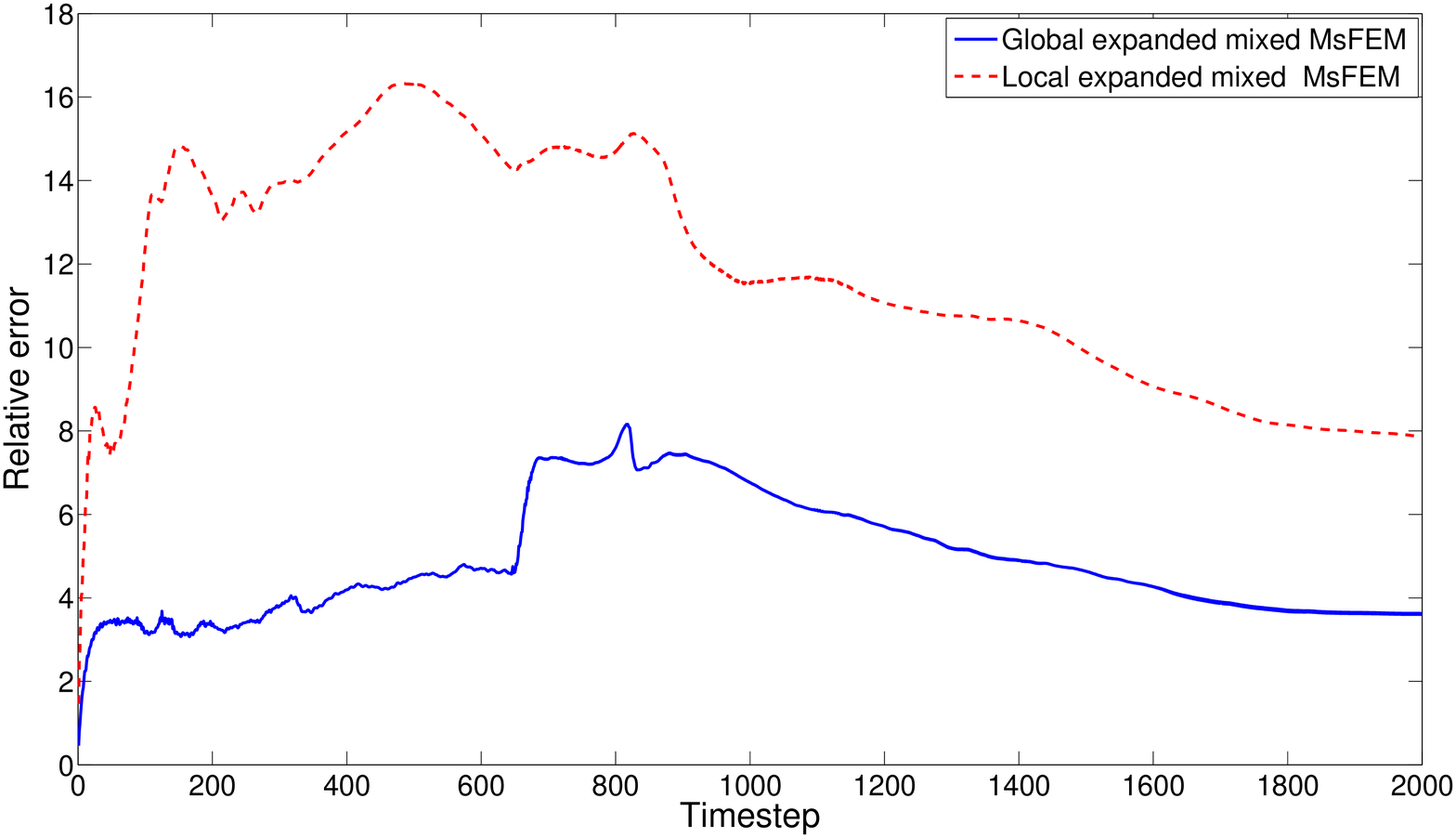}
  \end{center}
  \caption{Relative saturation errors (percentage) at each timestep of
    the IMPES simulation, for global and local expanded mixed MsFEM}
  \label{saterr}
\end{figure}

\begin{figure}
  \begin{center}
  \hspace*{3pt}
  \includegraphics[height=1.48in]{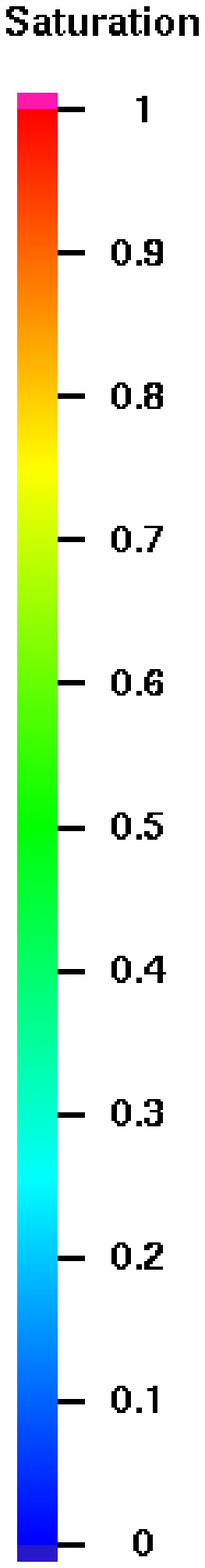}
  \includegraphics[height=1.48in]{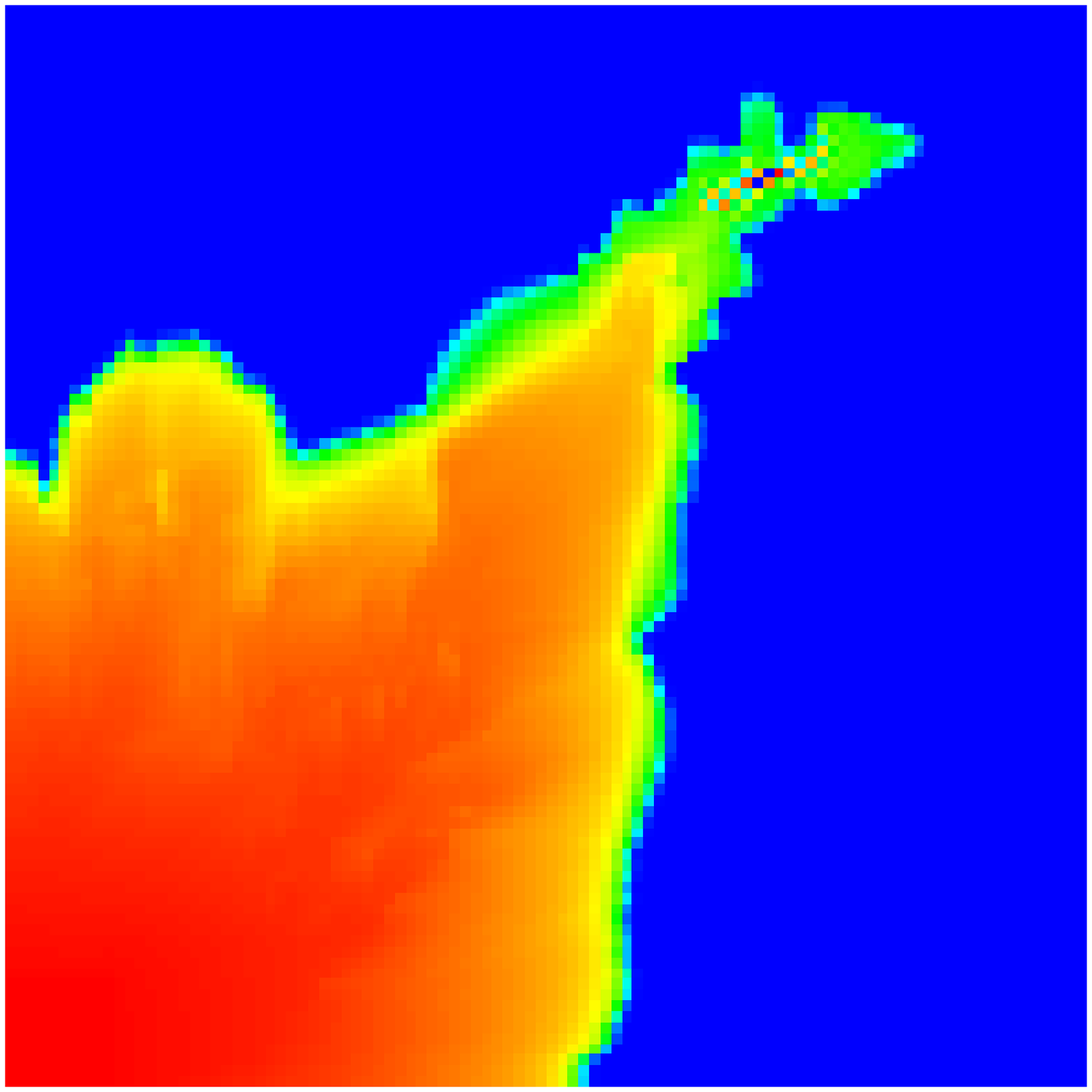}
  \includegraphics[height=1.48in]{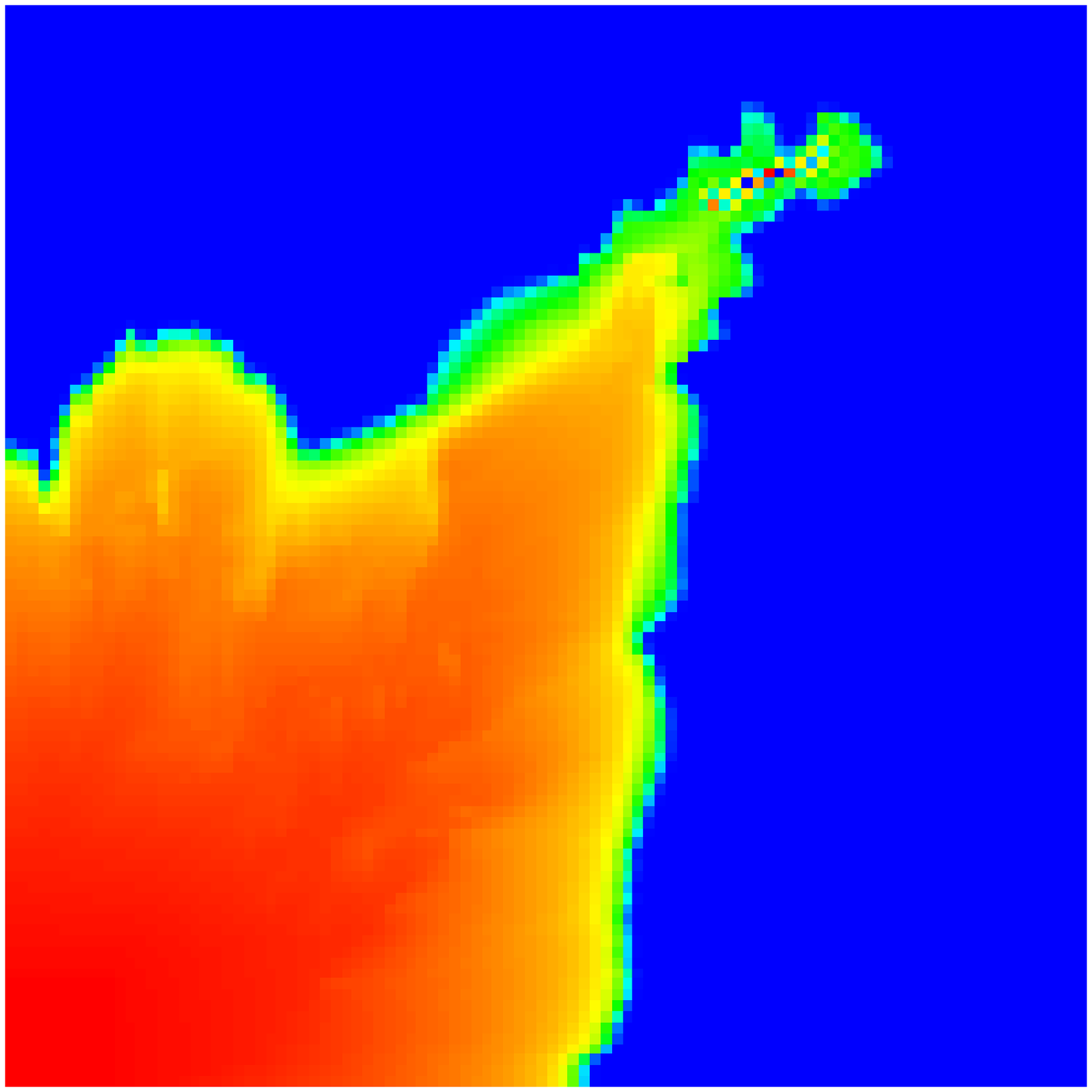}
  \hspace*{3pt}
  \includegraphics[height=1.48in]{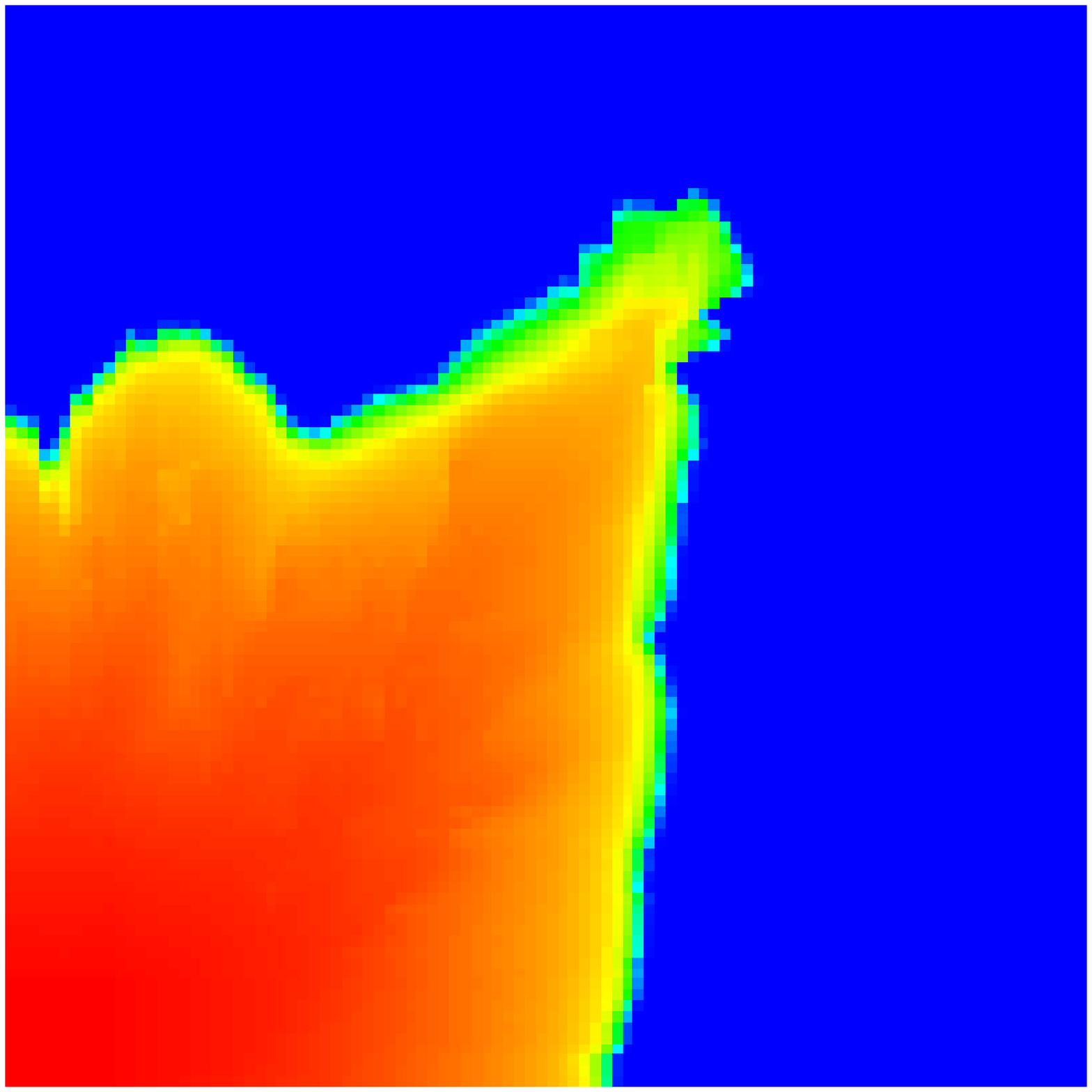}
  \end{center}
  \caption{Water saturation at the 700th timestep computed with the
    fine-scale reference velocity (left); with the coarse global
    expanded mixed MsFEM velocity (middle); and with the coarse local
    expanded mixed MsFEM velocity (right).}
  \label{fsat700}
\end{figure}

The two-phase flow scenario that we consider is the traditional
two-spot problem on a square domain (e.g., \cite{aej08, jennylt03}),
in which the water is injected at the bottom left corner and oil is
produced at the top right corner. Here, we use simple fixed rate wells
that are constant over a coarse cell.  We note that it is possible to
define a well on a fine cell, and capture the influence of the
near-well heterogeneity, by modifying the source term in the
construction of the basis functions.  Similar modifications were
developed to address this issue in the subgrid upscaling method
\cite{arbogast02} and the standard mixed MsFEM \cite{akl06}.
However, fine-scale wells are beyond the scope of this paper.


This two-spot problem is essentially two dimensional,
with the fine-grid discretization having $h=1/100$ ($100\times 100\times 8$
mesh) and $H=1/10$ or $1/20$.  The logarithm of the permeability
field is plotted in Figure \ref{logperm} and includes
features that are representative of shale barriers and sandy
deposits or channels.  Specifically, the darkest blue cells
have a  permeability of $10^{-6}$, but are masked by the
vanishing function given in \eqref{vanishingperm} on the reference cell
to represent a shale barrier.   In contrast, the red cells have
a permeability of $100$  and represent sandy deposits or channels.
The remainder of the fine-scale cells are assigned a moderate or low
permeability, and are shown as two distinct colors.

The water and oil mobilities are defined simply as the
quadratic functions $\lambda_w (S_w) = \frac{S_w^2}{\mu_w}$ and
$\lambda_o (S_w) = \frac{(1-S_w)^2}{\mu_o}$, respectively. Here,
$\mu_\alpha$ denotes the viscosity of phase $\alpha$, and we take
$\mu_w =0.5$, $\mu_o=1$. The total mobility $\lambda(S_w):=\lambda_w
(S_w)+\lambda_o (S_w)$. Then the pressure equation is given (in the
absence of gravity and capillary effects) by
\[
-div\big(\lambda(S_w)k\nabla p\big)=q.
\]
At each timestep, we use the velocity $u$ computed from the pressure
equation to advance the water saturation $S_w$ by an explicit
discretization of the saturation equation
\begin{equation*}
\frac{\partial S_w}{\partial t}+ div(f_w(S_w)u) = 0, \qquad\text{with }f_w(S_w)=
\frac{\lambda_w(S_w)}{\lambda_w(S_w)+\lambda_o(S_o)}.
\end{equation*}

Figure \ref{saterr} shows a plot of the saturation error in the
$L^2(\Omega)$ norm (relative to the reference IMPES saturation
computed with the fine-scale expanded mixed FEM velocity), at all
$2000$ time steps.  The water saturation, $S_w$, at the $700$th
timestep is plotted in Figure \ref{fsat700} by using the fine-scale
velocity and the global and local expanded mixed MsFEM velocities.
Figure \ref{fsat700} shows that the IMPES solution based on the global
expanded mixed MsFEM velocity remains close to the reference solution
throughout the long simulation, whereas the local expanded mixed MsFEM
velocity results in substantial error very early in the
simulation. The error naturally decreases later in the simulation, when
most of the domain becomes highly saturated with water.

The IMPES experiments discussed above used coarse velocity fields on a
$10\times 10\times 1$ coarse mesh ($H=1/10$), with each coarse cell
discretized by $10\times 10\times 8$ fine cells.  To study the coarse
velocity error behavior, we also computed the coarse velocity on a
$20\times 20\times 1$ coarse mesh, with each coarse cell discretized
by $5\times 5\times 8 $ fine cells. The results are reported in Table
\ref{randomvanish100}. The global expanded mixed MsFEM velocity is
about $25\%$ more accurate on the $20\times 20\times 1$ coarse mesh
than on the $10\times 10\times 1$, while the local expanded mixed
MsFEM velocity improves by less than $2\%$. Without global
information, non-local effects result in large errors even after
refinement.

\begin{table}
\caption{Randomly vanishing permeability without scale separation, $100\times 100\times 8$ fine mesh}
\label{randomvanish100}
\vspace*{-12pt}
\begin{center}
\begin{tabular}{|c|c|c|c|c|}
\hline
& \multicolumn{2}{|c|}{$20\times 20\times 1$ coarse mesh} & \multicolumn{2}{|c|}{$10\times 10\times 1$ coarse mesh} \\
\hline
& Exp. MsFEM & Exp. GMsFEM & Exp. MsFEM & Exp. GMsFEM \\
\hline \hline
Velocity error & $2.184\times 10^{-3}$ & $5.438\times 10^{-4}$ & $2.218\times 10^{-3}$ & $7.137\times 10^{-4}$ \\ \hline
Velocity norm & $6.202\times 10^{-3}$ & $6.202\times 10^{-3}$ & $6.202\times 10^{-3}$ & $6.202\times 10^{-3}$ \\ \hline
$\nabla p$ error & $3.529\times 10^{-3}$ & $1.999\times 10^{-3}$ & $3.636\times 10^{-3}$ & $2.283\times 10^{-3}$ \\ \hline
$\nabla p$ norm & $1.471\times 10^{-2}$ & $1.471\times 10^{-2}$ & $1.471\times 10^{-2}$ & $1.471\times 10^{-2}$ \\ \hline
\end{tabular}
\end{center}
\end{table}

We conclude that the expanded mixed MsFEM method is very effective in
approximating both the pressure gradient and velocity when a locally
vanishing permeability field makes standard mixed MsFEM methods
infeasible.  Practical simulations of two-phase flow demonstrate the
high degree of accuracy attained by using global information in
constructing the coarse basis functions. Thus the expanded mixed
formulation with global information is the most robust and accurate
method considered in this study for applications involving a locally
vanishing permeability field.

\section{Conclusions}

We developed a family of expanded mixed MsFEMs for elliptic equations
and considered their hybrid formulation.  In this formulation, the
four unknowns were solved simultaneously, namely, the pressure,
gradient of pressure, Lagrange multipliers, and velocity.  The
expanded mixed MsFEMs work well for the case that the coefficient
(e.g., permeability) is very small or locally vanishing in some
regions of the domain.  This case has important applications in
reservoir simulation.  In contrast the standard mixed MsFEMs require
the coefficient to be uniformly positive everywhere.  In addition,
expanded mixed MsFEMs provide the gradient of pressure without
significantly increasing the computational cost relative to the standard
method.

We analyzed the expanded mixed MsFEMs for separable scales and
non-separable scales, and established a priori error estimates.
We showed that the convergence rates of the local expanded MsFEMs
depend on both the small physical scales and on the coarse-mesh size.  In
contrast, we showed that the global expanded mixed MsFEMs achieves a
convergence rate that only depends on the coarse-mesh size.  To
support this analysis we applied the expanded mixed MsFEMs to various
models of flow in heterogeneous porous media.  For the models with
non-separable spatial scales, significantly better accuracy was
achieved by using global information to construct the multiscale basis
functions.  Moreover, for the periodic example of separable scales
using global information eliminated the resonance errors observed in
the the local expanded mixed MsFEM.
These numerical results showed the efficiency and robustness of the
developed methods.

\section*{Acknowledgments}
We would like to thank the referees for their comments and suggestions
to improve the paper.

\appendix
\section{proof of Theorem \ref{converg-os-thm}}
\label{app2}

Due to Theorem \ref{equi-thm-os}, problem (\ref{exp-fem-hybrid-os})
is equivalent to problem (\ref{exp-fem-num-os}), and we can use
problem (\ref{exp-fem-num-os}) to estimate the errors for
$((\theta_h^{os}, u_h^{os}), p_h^{os} )$.  By the estimate
(\ref{os-ineq0}), we have
\begin{eqnarray}
\label{os-ineq1}
\begin{split}
&\sum_K \|\big((\se, \ue), \pe\big)-\big((\theta_h^{os}, u_h^{os}), p_h^{os}\big)\|_{X(K)\times Q(K)}\\
&\leq C\big\{ \inf_{(\xi_h, v_h, q_h)\in X_{1,h}^{os}\times X_{2,h}^{os} \times Q_h} \sum_{K} \|((\se, \ue), \pe)-((\xi_h, v_h), q_h)\|_{X(K)\times Q(K)}  \big\}\\
& +\sum_{v_h\in X_{2,h}^{os}\setminus \{0\}} \frac{-(\se, v_h)-\sum_{K} (div(v_h), \pe)}{\sum_K \|v_h\|_{div, K}}\\
&:=I_1 +I_2.
   \end{split}
\end{eqnarray}
We set $q_h=\mathcal{P}_h \pe$, $\xi_h|_K=\nabla \wes|_K$ and
$v_h=\mathcal{I}_h^{os} u^*$ in (\ref{os-ineq1}).
Lemma~\ref{vh-eq-os} shows that $\xi_h\in X_{1,h}^{os}$ and $v_h\in
X_{2,h}^{os}$.  By the Lemma 4.6 in \cite{ch03}, we obtain
\begin{eqnarray}
\label{os-ineq2}
|I_1| \leq C(h+\epsilon) (\|f\|_{1,\Omega}+\|p^*\|_{2,\Omega}) +C({\epsilon \over h}+\sqrt{\epsilon})\|p^*\|_{1,\infty, \Omega}.
\end{eqnarray}
By using the proof of Theorem 2.2 in \cite{ch03}, we estimate the
consistence error $I_2$ by
\begin{eqnarray}
\label{os-ineq3}
|I_2| \leq C(h+\epsilon) (\|f\|_{1,\Omega}+\|p^*\|_{2,\Omega}) +C({\epsilon \over h}+\sqrt{\epsilon})(\|p^*\|_{1,\infty, \Omega}+\|f\|_{0, \Omega}).
\end{eqnarray}

Next we estimate $\|\mathcal{P}_{\partial} \pe
-\lambda_h^{os}\|_{-{1\over 2}, h}$.  Let
$\tilde{\tau}_h:=|e|(\mathcal{P}_{\partial} \pe
-\lambda_h^{os})\bar{\psi}_{\chi_e}^K\in \tilde{X}_{2,h}^{os}(K)$ and
the operator $\mathcal{M}_h$ be defined in (\ref{def-Mh-os}). Then
\[
\mathcal{M}_h (\tilde{\tau}_h)\cdot n_e=\mathcal{P}_{\partial} \pe -\lambda_h^{os} \quad  \text{on $e$ and $0$ otherwise}.
\]
By direct calculation or scaling argument, it follows
\begin{equation}
\label{mult-00-os}
\|\mathcal{M}_h(\tilde{\tau}_h)\|_{0,K} +h_K \|div \mathcal{M}_h (\tilde{\tau}_h)\|_{0, K} \leq C h_K^{1\over 2} \|\mathcal{P}_{\partial} \pe -\lambda_h^{os}\|_{0,e}.
\end{equation}
Define $\tau_h=\tilde{\tau}_h$ in K and $\tau_h=0$ in
$\Omega\setminus$ K.  By the second equation in
(\ref{exp-fem-hybrid-os}), we have
\begin{equation*}
(\theta_h^{os}, \tilde{\tau}_h)_K +(p_h^{os}, div\tilde{\tau}_h)_K=(\mathcal{P}_{\partial} \pe -\lambda_h^{os}, \lambda_h^{os})_e.
\end{equation*}
Since $div\mathcal{M}_h(\tilde{\tau}_h)=div\tilde{\tau}_h$, it follows
that
\begin{equation}
\label{mult-01-os}
(\theta_h^{os}, \tilde{\tau}_h)_K +(p_h^{os}, div\mathcal{M}_h(\tilde{\tau}_h))_K=(\mathcal{P}_{\partial} \pe -\lambda_h^{os}, \lambda_h^{os})_e .
\end{equation}

Since $\se=\nabla \pe$, Green's function gives
\begin{equation}
\label{mult-02-os}
(\se, \mathcal{M}_h(\tilde{\tau}_h))_K +(\pe, div\mathcal{M}_h(\tilde{\tau}_h))_K=(\mathcal{P}_{\partial} \pe -\lambda_h^{os}, \pe)_e.
\end{equation}
By a similar argument to (\ref{inf-sup-eq2}), it follows that
\begin{equation}
  \label{inverse-Mh-bd}
  \|v_h\|_{0,K} \leq C\|\mathcal{M}_h v_h\|_{0,K} \quad \text{for} \quad \forall v_h\in\tilde{X}_{2,h}^{os}(K).
\end{equation}
By using (\ref{mult-01-os}), (\ref{mult-02-os}), (\ref{mult-00-os})
and (\ref{inverse-Mh-bd}), we get
\begin{eqnarray}
\begin{split}
&\|\mathcal{P}_{\partial} \pe -\lambda_h^{os}\|_{0,e}^2 =(\mathcal{P}_{\partial} \pe -\lambda_h^{os}, \mathcal{P}_{\partial} \pe -\lambda_h^{os})_e=( \pe -\lambda_h^{os}, \mathcal{P}_{\partial} \pe -\lambda_h^{os})_e\\
&=(\se-\theta_h^{os}, \tilde{\tau}_h)_K+ (\se, \mathcal{M}_h(\tilde{\tau}_h)-\tilde{\tau}_h)_K   +   (\pe-p_h^{os}, div\mathcal{M}_h(\tilde{\tau}_h))_K\\
&\leq \|\se-\theta_h^{os}\|_{0,K} \|\tilde{\tau}_h\|_{0,K} + \|\se\|_{0,K}( \|\mathcal{M}_h(\tilde{\tau}_h)\|_{0,K} + \|\tilde{\tau}_h\|_{0,K})\\
&+  \|\pe-p_h^{os}\|_{0,K} \|div\mathcal{M}_h(\tilde{\tau}_h)\|_{0,K}\\
&\leq C\|\se-\theta_h^{os}\|_{0,K}  \|\mathcal{M}_h(\tilde{\tau}_h)\|_{0,K}+C\|\se\|_{0,K}\|\mathcal{M}_h(\tilde{\tau}_h)\|_{0,K}+ \|\pe-p_h^{os}\|_{0,K} \|div\mathcal{M}_h(\tilde{\tau}_h)\|_{0,K}\\
&\leq C\big( h_K^{-{1\over 2}}\|\pe-p_h^{os}\|_{0,K}+h_K^{1\over 2} \|\se-\theta_h^{os}\|_{0,K}+h_K^{1\over 2}\|\se\|_{0,K} \big) \|\mathcal{P}_{\partial} \pe -\lambda_h^{os}\|_{0,e},
\end{split}
\end{eqnarray}
which gives
\[
\|\mathcal{P}_{\partial} \pe -\lambda_h^{os}\|_{0,e}\leq C\big( h_K^{-{1\over 2}}\|\pe-p_h^{os}\|_{0,K}+h_K^{1\over 2} \|\se-\theta_h^{os}\|_{0,K} +h_K^{1\over 2}\|\se\|_{0,K}  \big).
\]
Consequently,
\begin{eqnarray}
\label{mult-03-os}
\begin{split}
&\|\mathcal{P}_{\partial} \pe -\lambda_h^{os}\|_{-{1\over 2}, h}^2=\sum_{K\in  \mathfrak{T}_h}  \sum_{e\in \subset \pK} h_K \|\mathcal{P}_{\partial} \pe -\lambda_h^{os}\|_{0,e}^2\\
&\leq C \|\pe-p_h^{os}\|_{0, \Omega}^2 +Ch^2 \|\se-\theta_h^{os}\|_{0, \Omega}^2+Ch^2 \|\se\|_{0,K}.
\end{split}
\end{eqnarray}
Combining (\ref{os-ineq1}), (\ref{os-ineq2}), (\ref{os-ineq3}) and (\ref{mult-03-os}) proves  (\ref{error-fgp-os}).

\section{proof of Theorem \ref{w-lem-g}}
\label{app3}

By straightforward calculation, it follows that
\begin{eqnarray}
\label{eq01-g}
\begin{split}
\begin{cases}
div[\sum_{i=1}^N  \sum_{e\subset \pK}(\int_e A_iu_i\cdot n_eds) (-k\nabla \phi_{i,\chi_e}^K)\big]&=div (\mathcal{I}_h^g u|_K) \quad \text{in $K$}\\
\sum_{i=1}^N \sum_{e\subset \pK}(\int_e A_iu_i\cdot n_eds) (-k\nabla \phi_{i,\chi_e}^K)\cdot n&=(\mathcal{I}_h^g u|_K)\cdot n  \quad \text{on} \quad \pK.
\end{cases}
\end{split}
\end{eqnarray}
Comparing equation (\ref{w-eq-global}) with equation (\ref{eq01-g}), we find that
\[
-k\nabla w=\sum_{i=1}^N  \sum_{e\subset \pK}(\int_e A_iu_i\cdot n_eds) (-k\nabla \phi_{i,\chi_e}^K)=\mathcal{I}_h^g u|_K.
\]
By the uniqueness of solution for (\ref{w-eq-global}), it follows that
\[
\nabla w=\sum_{i=1}^N  \sum_{e\subset \pK}(\int_e A_iu_i\cdot n_eds)\eta_{i, \chi_e}^K\in X_{1,h}^g(K) .
\]


\end{document}